\documentclass[letterpaper,11 pt]{article}
\usepackage{amsfonts,amsmath,graphicx,fullpage,bbm,amsthm}
\usepackage{epsfig,subfigure,float,amssymb,multirow,verbatim}
\usepackage{acronym,wrapfig,plain,mathrsfs,enumerate,relsize,color}
\usepackage{soul}
\usepackage{tikz}
\usetikzlibrary{arrows}
\usepackage[latin1]{inputenc}

\def\an#1{{{\color{black}#1}}}
\def\us#1{{{\color{black}#1}}}

\def\uvs#1{{{\color{black}#1}}}
\def\jk#1{{{\color{black}#1}}}
\def\jk#1{{{\color{black}#1}}}
\def\fskip#1{}

\newtheorem{assumption}{Assumption}

\newtheorem{corollary}{Corollary}

\newtheorem{example}{Example}

\newtheorem{lemma}{Lemma}

\newtheorem{proposition}{Proposition}

\def\Gscr{{\cal G}}
\def\Nscr{{\cal N}}
\def\Escr{{\cal E}}

\def\be{\begin{enumerate}}
\def\ee{\end{enumerate}}
\def\Nscr{{\cal N}}
\newcommand{\pmat}[1]{\begin{pmatrix} #1 \end{pmatrix}}
\def\ren{\mathbb{R}^n}
\def\re{\mathbb{R}}

\def\argmin{\mathop{\rm argmin}}

\title{Distributed Algorithms for Aggregative Games on Graphs}

\author{Jayash Koshal, Angelia Nedi\'c and Uday
	V.~Shanbhag\thanks{Koshal and Nedi\'{c} are in the 
Department of Industrial and Enterprise
Systems Engineering,
University of Illinois, Urbana IL 61801, while Shanbhag is in the
	Department of Indsutrial and Manufacturing Engineering, Pennsylvania
	State University, University Park, PA 16802. They are contactable at
{\{koshal1,angelia\}}@illinois.edu and {udaybag@psu.edu}. Nedi\'c and Shanbhag
	gratefully acknowledge the NSF support of this work through the
grants NSF CMMI 09-48905 EAGER ARRA (Nedi\'{c} and Shanbhag), NSF CAREER
CMMI 1246887 (Shanbhag),  and NSF CMMI 07-42538 (Nedi\'{c}).}
}

\begin{document}
\maketitle
\thispagestyle{empty}
\pagestyle{empty}

\begin{abstract}
 \noindent
   We consider a class of Nash games, termed as aggregative games, being
   played over a networked system. In an aggregative game, a player's
   objective is a  function of the aggregate of all the players'
   decisions.  Every player maintains an estimate of \uvs{this}
   aggregate, and the players exchange this information with their
   \uvs{local} neighbors over a connected network. We study distributed
   synchronous and asynchronous algorithms for information exchange and
   equilibrium computation over such a network. Under standard
   conditions, we establish the almost-sure convergence of the obtained
   sequences to the equilibrium point.  We also consider extensions of
   our schemes to aggregative games where the players' objectives are
   coupled through a more general form of aggregate function.  Finally,
   we present numerical results that demonstrate the performance of the
   proposed schemes.
\end{abstract}


\section{Introduction}
An aggregative game is a non-cooperative Nash  game in which each
player's payoff depends on its action and an aggregate
function of the actions taken by all
players~\cite{novshek_courexist_1985,
	jensen_agggames_2006,jensen10aggregative,
	martimort_aggrepresent_2011}.\footnote{Such games have been shown to be
	closely related with subclasses of potential
	games~\cite{Schipper04pseudo-potentialgames,dubey_pseudopotential06}
	where a potential game refers to a Nash game in which the payoff
	functions admit a potential function~\cite{MondererShapley96}.
	\uvs{The potential function of an aggregative game is a {special case
		of the function employed} in \cite{consensus_propagation}, where
			distributed algorithms for optimization problems with general separable convex functions are presented.}}
	Nash-Cournot
games represent an important instance of such games; here, firms make quantity
bids that fetch a price based on aggregate quantity sold, implying
that the payoff of any player is a function of the aggregate
sales~\cite{cournot1838,fudenberg91game}. The ubiquity of
	such games has grown immensely in the last two decades and examples
		emerge in the form of supply function
		games~\cite{jensen10aggregative}, common agency
		games~\cite{jensen10aggregative}, and power and rate control in
		communication
		networks~\cite{alpcan02game,alpcan03distributed,basar07control,yin09nash2}
	(see~\cite{ferrer05evolutionary} for more examples).
		Our work is motivated by the development of distributed
		algorithms on a range of  game-theoretic problems in
wired and wireline communication networks where such an aggregate function
captures the link-specific
congestion~\cite{alpcan03distributed,basar07control} or the
signal-to-noise
ratio~\cite{pan09games,pan10system,stefanovic11lyapunov}. In almost
all of the algorithmic research on equilibrium computation, it is assumed that  the aggregate of player decisions
	is observable to all players, allowing every player to evaluate its
	payoff function without any prior communication. 

In this paper, we consider aggregative games wherein the players 
(referred to as agents) compete over a  network.  
Distributed computation of equilibria in such games is complicated by
two crucial challenges. First, the connectivity graphs of the underlying
network may evolve over time. Second, agents do not have ready access to
aggregate decisions, implying that agents cannot compute their payoffs
(or their gradients). Consequently, distributed
gradient-based~\cite{alpcan03distributed,pan09games,kannan10online,koshal10single}
or best-response schemes~\cite{scutari13joint} cannot be directly
implemented since agents do not have immediate access to the aggregate.
Accordingly, we propose two distributed agreement-based algorithms
{that overcome this difficulty by 
allowing agents to build estimates of the aggregate by communicating
	with their local neighbors} and consequently
compute an equilibrium of aggregative games. Of these, the first is a
synchronous algorithm where all agents update simultaneously, while the
second, a gossip-based algorithm, allows for asynchronous computation:
\begin{enumerate}
\item[(a)] {\em {Synchronous}  distributed  algorithm:} At each epoch,
	every agent performs a ``learning step'' to update its estimate of
	the aggregate using the information obtained through the
	time-varying states of its neighbors.  All agents exchange
	information and subsequently update {their decisions}
	simultaneously {via a gradient-based update}.  This
	algorithm builds on the ideas of the method developed
	in~\cite{ram_2010_stocsub} for distributed optimization problems.
\item[(b)] {\em Asynchronous distributed algorithm:} In contrast, the
asynchronous algorithm uses a gossip-based protocol for information
exchange.  In the gossip-based {scheme}, a single pair of randomly
selected {neighboring agents}
	exchange their information and update their estimates of \us{both} the
	aggregate and their \us{individual} decisions.  \us{This} algorithm combines our
	synchronous method in (a) with the gossip technique proposed
	in~\cite{boyd_gossip_2006} for the agreement (consensus) problem.
\end{enumerate}
We investigate the convergence behavior of both algorithms under a
diminishing stepsize rule, and provide error bounds under a constant
	steplength regime. Additionally, the results are supported with
	numerics derived from application of the proposed schemes on a class
	of networked Nash-Cournot games. The novelty of this work lies in our
	examination of distributed (neighbor-based) algorithms for
	computation of a Nash equilibrium point for aggregative Nash games,
	while the majority of preceding efforts on such algorithms have been
	spent towards solving feasibility and optimization problems.
        Before proceeding, a caveat is in order. While the proposed
	game-theoretic problem can be easily solved via a range of
		centralized schemes (see~\cite{Facchinei_Pang03} for a
				\an{comprehensive survey}), any such
		approach relies on the centralized availability of all
		information, a characteristic that does {\em not} hold in the
		present setting.  Instead, our interest is not merely in
		equilibrium computation but in the development of stylized
		distributed protocols, implementable on networks, and
		complicated by informational restrictions, local communication
		access, and a possibly evolving network structure.

Broadly speaking, the present work can be situated in the larger domain of distributed computation of equilibria
in networked Nash games. First proposed by Nash in
	1950~\cite{nash50equilibrium}, this equilibrium concept has found
	application in modeling strategic interactions in  oligopolistic
	problem settings drawn from economics, engineering, and the applied
	sciences~\cite{fudenberg91game,facchinei09variational,basar07control}.
	More recently, game-theoretic models have assumed relevance in the
	control of a large collection of coupled nonlinear
	systems, instances of which arise in production
	planning~\cite{huang07large-population}, synchronization of
	coupled oscillators~\cite{yin12synchronization}, amongst others. In
	particular, agents in such settings have conflicting objectives and
	the centralized control problem is challenging. By allowing
	agents to compete, the equilibrium behavior may be analyzed exactly
	or approximately (in large population regimes), allowing for the
	derivation of distributed control laws. In fact, game-theoretic
	approaches have been effectively utilized in obtaining distributed
	control laws in complex engineered systems~\cite{ran08,ran09}.
	Motivated by the ubiquity of game-theoretic models, arising either
	naturally or in an engineered form, the distributed computation of
	equilibria has immense importance.

	While \us{equilibrium} computation is a well-studied
	topic~\cite{Fudenberg}, our interest lies in networked regimes where
	agents can only access or observe the decisions of their neighbors.
	In such contexts, our interest lies in developing distributed
	gradient-based schemes. While any such algorithmic development is
	well motivated by protocol design in networked multi-agent systems,
	best-response schemes, rather than gradient-based methods, are
	natural choices when players are viewed as fully rational.
	However, gradient-response schemes assume relevance for several
	reasons. First, increasingly game-theoretic approaches are
	being employed for developing distributed control protocols where
	the choice of schemes lies with the designer
	(cf.~\cite{li13designing,marden13modelfree}). Given the relatively
	low complexity of gradient updates, such avenues are attractive for
	control systems design.  Second, when players rule out
	strategies that are characterized by high computational
	complexity~\cite{Papadimitriou94onbounded} (referred to as a
			``bounded-rationality'' setting\footnote{This notion is
			rooted in the influential work by
			Simon~\cite{Simon96sciences} where it  is suggested that,
			when reasoning and computation are costly, agents may not
			invest in these resources for marginal benefits.}),
	gradient-based approaches become relevant  and have been
		employed extensively in the context of communication
			networks~\cite{alpcan02game,alpcan03distributed,yin09nash2,pan09games,stefanovic11lyapunov}.

The present work assumes a strict monotonicity property on the
mapping corresponding to the associated variational problem. This
assumption is weaker than that imposed by related studies on
communication networks~\cite{alpcan02game,alpcan03distributed} where
strong monotonicity properties are imposed. From a methodological
standpoint, we believe that this work represents but a first step.  By
	combining a regularization technique, this requirement can be
	weakened~\cite{kannan10online} while extensions to stochastic
	regimes can also be incorporated by examining regularized
	counterparts of stochastic approximation~\cite{koshal10single}.
	However, all of these approaches are under the assumption that
	agents have access to the decisions of all their competitors.

Finally, it should be emphasized that the distributed algorithms
presented in this paper draw inspiration from the seminal work
in~\cite{Tsitsiklis84}, where a distributed method for optimization has
been developed by allowing agents to communicate locally with their
neighbors over a time-varying communication network.  This idea has
attracted a lot of attention recently in an effort to extend the
algorithm of~\cite{Tsitsiklis84} to more general and broader range of
problems~\cite{nedic09distributed,Ram2010,Ram2009cdc,Johansson2009,Johansson_thesis,
nedic10constrained,Kunal2011,Nedic2011,
lobel11distributed,tu12diffusion,ram2012,tu12influence,Bianchi2012}.
Much of the aforementioned work focuses on optimizing the
sum of local objective 
function~\cite{nedic09distributed,Ram2010,Ram2009cdc,
nedic10constrained,Kunal2011,Nedic2011,lobel11distributed} 
in a multi-agent networks, while a subset of recent work 
considered the min-max 
optimization problem~\cite{kunal_bregman_11,Boche_minmax_2007}, 
where the objective is to 
minimize the maximum cost incurred by any agent in the network. 
Notably, extensions of consensus based 
algorithms have also 
been studied in the domain of distributed 
regression~\cite{ram2012}, estimation and 
inference 
tasks~\cite{tu12diffusion,tu12influence}. \jk{While \uvs{much of} the aforementioned work focuses on consensus-based algorithms, an 
alternative distributed messaging protocol \uvs{for}
consensus propagation 
	across a network is presented in~\cite{min_sum_message_passing}.} The work in 
this 
paper extends the realm of \jk{consensus based algorithms (and not consensus propagation)} to capture 
competitive aspect of multi-agent networks.

The remainder of the paper is organized as follows. In
section~\ref{sec:model}, we describe the problem of interest, provide
two motivating examples and state our assumptions. A synchronous
distributed algorithm is proposed in section~\ref{sec:synchron} and
convergence theory is provided. An asynchronous gossip-based variant of
this algorithm is described in section~\ref{sec:asynchron} and is
supported by convergence theory and error analysis.  In
section~\ref{sec:extension}, we present an extension to the problem
presented in section~\ref{sec:model} and suitably adapt the distributed
synchronous and asynchronous algorithm to address this generalization.
We present some numerical results in section~\ref{sec:numerics} and,
   finally, conclude in section~\ref{sec:conclusion}. 

Throughout this paper, we view vectors as columns. We write $x^T$
to denote the transpose of a vector $x$, and $x^Ty$
to denote the inner product of vectors $x$ and $y$.
We use  $\|x\|=
\sqrt{x^Tx}$ to denote the Euclidean norm of a vector $x$.
We use $\Pi_K$ to denote the Euclidean projection operator on a set $K$,
i.e.\ $\Pi_K(x)\triangleq \argmin_{z \in K} \|x-z\|.$ The
expectation of a random variable $Y$ is denoted by $\mathbb{E}[Y]$ and {\it
		a.s.} denotes {\it almost surely}.
		A matrix $W$ is row-stochastic if $W_{ij}\ge0$ for all $i,j$, and $\sum_{j}W_{ij}=1$.
A matrix $W$ is doubly stochastic if both $W$ and $W^T$ are \jk{row} stochastic.

\section{Problem Formulation and Background}  \label{sec:model}
In this section we introduce an aggregative game of our interest and
provide its sufficient equilibrium conditions. The players in this
	game are assumed to have local interactions with each other
	over time, where these interactions are modeled by time-varying
	connectivity graphs.  We also discuss some auxiliary results for the
	players' connectivity graphs and present our distributed algorithm
	for equilibrium computation.

\subsection{Formulation and Examples}\label{sec:game}
Consider a set of $N$ players (or agents) indexed by $1,\ldots,N,$ and
let $\Nscr = \{1,\ldots,N\}$.  The $i$th player is characterized by a
strategy set $K_i \subseteq\mathbb{R}^n$ and a payoff function
$f_i(x_i,\bar x)$, which depends on player $i$ decision $x_i$ and the
aggregate $\bar x=\sum_{i=1}^N x_i$ of all player decisions.
\uvs{Furthermore, $\bar x_{-i}$ denotes the aggregate of all players
	excepting player $i$, i.e.,
\[
\bar x_{-i} = \sum_{j=1, j \neq i}^Nx_j.
\]} To
formalize the game, let $\bar K$ denote the Minkowski sum of the sets
$K_i$, \us{defined as follows}:
\begin{equation}\bar K\triangleq \sum_{i=1}^N K_i.\label{eqn:setsumk}
\end{equation}
In a generic aggregative game, \us{given $\bar x_{-i}$}, player $i$  faces the following
parametrized optimization problem:
\begin{align}\label{eqn:problem}
\mbox{minimize} & \qquad f_i(x_i,\bar x) \, \triangleq\, f_i(x_i, x_i + \bar x_{-i}) \cr
\mbox{subject to} & \qquad x_i \in K_i,
\end{align}
where $K_i \subseteq\mathbb{R}^n$ and $\bar x$ is the aggregate of the agent's decisions $x_i$, i.e., 
\begin{align}\label{eqn:aggreg}
\bar x \, := \, \sum_{j=1}^Nx_j = x_i + \bar x_{-i},\qquad \bar x\in \bar K,
\end{align}
with $\bar K\subseteq\ren$ as given in~\eqref{eqn:setsumk}, and
$f_i:K_i\times \bar K\to\re$. 
The set $K_i$ and the function $f_i$ are assumed to be known by agent
$i$ only. Next, we motivate our work by providing an example of an aggregative
game, whose broad range emphasizes the potential scope of our work.

\begin{example}[Nash-Cournot game over a network]
\label{ex:nash_example}
A classical example of an aggregative game is a Nash-Cournot
played over a network~\cite{cournot1838,okuguchi90theory,kannan10online}. 
Suppose a set of $N$ firms compete over $\mathcal{L}$ locations. 
In this situation, the communication network of our interest 
is formed by the players which are viewed as  the nodes in the network. One such instance of connectivity graph is as shown in Figure~\ref{fig:connectivity_example}. This graph determines how the firms communicate their production decision over $\mathcal{L}$ locations. More specifically, the firm in the center of the graph has access to information from all the other firms, 
whereas all the other firms have access to the information of the firm in the center only. We consider other instances of connectivity graph in Section~\ref{sec:numerics}. To this end, let firm
\begin{figure}
\centering
\begin{tikzpicture}[scale=1.2]
\tikzstyle{vertex}=[draw,shape=circle];
\tikzstyle{selected}=[vertex,fill=black];
\tikzstyle{neighbor}=[vertex,fill=black!30];
\tikzstyle{nonneighbor}=[vertex,fill=black!0];

\tikzstyle{edge} = [draw,thick,-]
\tikzstyle{selected edge} = [draw,line width=1.2pt,dashed,black]
\tikzstyle{ignored edge} = [draw,line width=1.2pt,black]

\path[selected] (0:0cm)    node[neighbor] (v0) {};
\path (0:1cm)    node[neighbor] (v1) {};
\path (72:1cm)   node[neighbor] (v2) {};
\path (2*72:1cm) node[neighbor] (v3) {};
\path (3*72:1cm) node[neighbor] (v4) {};
\path (4*72:1cm) node[neighbor] (v5) {};


\path[ignored edge] (v0) -- (v1);
\path[ignored edge] (v0) -- (v2);
\path[ignored edge] (v0) -- (v3);
\path[ignored edge] (v0) -- (v4);
\path[ignored edge] (v0) -- (v5);
\end{tikzpicture}
\caption{An information connectivity graph for a network with 6 firms.}
\label{fig:connectivity_example}
\end{figure}
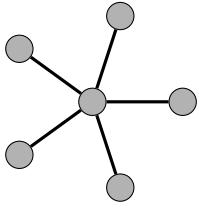
$i$'s production and sales 
at location 
$l$ \jk{be} denoted by
$g_{il}$ and $s_{il}$, respectively, while its cost of production at location $l$ is denoted by $c_{il}(g_{il})$. Consequently, goods sold by firm $i$ at location $l$ fetch a revenue
$p_l(\bar s_l)s_{il}$ where $p_l(\bar s_l)$ denotes the sales price
at location $l$ and $\bar s_l = \sum_{i=1}^N s_{il} $ represents the aggregate
sales at location $l$. Finally, firm $i$'s production at location $l$ is capacitated
by $\mathrm{cap}_{il}$ and its optimization problem is given by the
following\footnote{Note that the transportation costs are assumed to be
	zero.}:
\begin{align}\label{eqn:ncproblem}
\mbox{minimize} & \qquad \sum_{l =1}^\mathcal{L} \left( c_{il} (g_{il}) - p_l(\bar s_l)
		s_{il}\right)\cr
\mbox{subject to} & \qquad \sum_{l=1}^\mathcal{L} g_{il} {\geq} \sum_{l=1}^\mathcal{L} s_{il},
	\notag \\& \qquad g_{il}, s_{il} \geq 0,\quad g_{il} \leq \mathrm{cap}_{il}, \qquad l = 1,
	\hdots, \mathcal{L}. 
\end{align}
In effect, firm $i$'s payoff function is parametrized by nodal aggregate
sales, thus rendering  an aggregative game.  Note that, in 
this example we have two independent networks, the first being 
used to model the communication of the firms and the second 
being used to model the physical layout of the firms production 
unit and locations. We allow the communication network to be dynamic but the layout network 
is assumed to be static.
\end{example}

\subsection{Equilibrium Conditions and Assumptions}
To articulate sufficiency conditions, we make the following assumptions on the constraint sets 
$K_i$ and the functions $f_i.$ 

\begin{assumption}\label{assump:set_func} 
For each $i = 1,\ldots,N$,
	the set $K_i \subset{\mathbb{R}^n}$ is compact and convex. Each
		function $f_i(x_i,y)$ is continuously differentiable in
		$(x_i,y)$ over some open set containing the set $K_i\times \bar K$,
		while each function $x_i\mapsto f_i(x_i,\bar x)$ is convex over the set $K_i$. 
\end{assumption}

Under Assumption~\ref{assump:set_func}, the (sufficient) equilibrium
conditions of the Nash game in~\eqref{eqn:problem} 
can be specified as a variational inequality problem 
VI$(K,\phi)$ (cf.~\cite{Facchinei_Pang03}). Recall that
	VI$(K,\phi)$ requires determining a point $x^*\in K$
such that 
\[(x-x^*)^T\phi(x^*) \jk{\geq}  0 \qquad\hbox{for all }x\in K,\]
where
\begin{equation}\label{eqn:svikf}
 \phi(x) \triangleq
\left( \begin{array} {c}
       \nabla_{x_1} f_1(x_1,\bar x) \\
       \vdots \\
       \nabla_{x_N} f_N(x_N, \bar x)
      \end{array}
\right), \qquad K = \prod_{i=1}^{N} K_i,
\end{equation}
with $x\triangleq(x_1^T,\ldots,x_N^T)^T$, $x_i\in K_i$ for all
$i$, and $\bar x$ defined by~\eqref{eqn:aggreg}.  Note that, by
Assumption~\ref{assump:set_func}, the set $K$ is a compact and
convex set in $\mathbb{R}^{nN}$, and the mapping  $\phi: K \to\mathbb{R}^{nN}$
is continuous. To emphasize the particular form of the mapping $\phi$, 
we define $F_i(x_i, \bar x)$ as follows:
\begin{align}\label{eqn:compfi}
F_i(x_i,\bar x)=\nabla_{x_i} f_i(x_i,\bar x)\qquad\hbox{for all }i=1,\ldots,N.
\end{align}
The mapping $F(x,u)$ is given by 
\begin{equation}\label{eqn:svinot}
 F(x,u) \triangleq
\left( \begin{array} {c}
       F_1(x_1,u)\cr       
       \vdots \cr
       F_N(x_N, u)
      \end{array}
\right),
\end{equation}
where the component maps $F_i: K_i\times \bar K\to\ren$ are given by~\eqref{eqn:compfi}.
With this notation, we have
\begin{equation}\label{eqn:phimap}
\phi(x)=F(x,\bar x)\qquad\hbox{for all $x\in K$}.\end{equation}

Next, we make an assumption on the mapping $\phi(x)$.
\begin{assumption}\label{ass-strict-mon}
	The mapping $\phi(x)$ is strictly monotone over
	$K$, i.e., 
	$$ (\phi(x) - \phi(x'))^T(x-x') > 0, \quad \hbox{for all }x,x' \in
	K, \us{ \mbox{ where } x \neq x'}.$$
\end{assumption}
Assumption~\ref{assump:set_func} allows us to claim the existence of a
Nash equilibrium, while Assumption~\ref{ass-strict-mon} allows us to
claim the uniqueness of the equilibrium.

\begin{proposition}\label{uniq-equil}
Consider the aggregative Nash game defined in \eqref{eqn:problem}.
Suppose Assumptions~\ref{assump:set_func} and~\ref{ass-strict-mon} hold. 
Then, the game admits a unique Nash equilibrium.
\end{proposition}
\begin{proof}
	By Assumption~\ref{assump:set_func}, the set $K$ is compact and
		$\phi$ is continuous. It follows from Corollary
			2.2.5~\cite{Facchinei_Pang03} that  
	VI$(K,\phi)$  has a solution. By the strict
	monotonicity of $\phi(x)$, VI$(K,\phi)$ has at most one solution
	based on Theorem~2.3.3~\cite{Facchinei_Pang03} and uniqueness
	follows.
	\end{proof}

Strict monotonicity assumptions on the mapping are seen to hold in a range of
practical problem settings, including Nash-Cournot
games~\cite{kannan10online}, rate allocation
problems~\cite{alpcan02game,alpcan03distributed,yin09nash,yin09nash2}, amongst others.  We now state our assumptions on the mappings $F_i$, which are related 
to the coordinate mappings of $\phi$ in~\eqref{eqn:svikf}.
\begin{assumption} \label{assump:Lipschitz}
Each mapping $F_i(x_i,u)$ is uniformly Lipschitz continuous in $u$ over $\bar K$,
  for every fixed $x_{i} \in K_{i}$ i.e., for some $\bar L_{i}>0$ and
  for all $\us{z_1},\us{z_2}\in \bar K$,
  \[\|F_i(x_{i},\us{z_1})-F_i(x_{i}, \us{z_2})\| \leq \bar
  L_i\|\us{z_1}-\us{z_2}\|,\]
 where $\bar K$ is as defined in \eqref{eqn:setsumk}.
\end{assumption}

One would naturally question whether such assumptions are seen to hold
in practical instances of aggregative games. 
We will show in section~\ref{sec:numerics} that the assumptions
are satisfied for the  Nash-Cournot game described in Example~\ref{ex:nash_example}.

Before proceeding, it is worthwhile to reiterate the motivation for the
present work. In the context of continuous-strategy Nash games, when the
mapping $\phi$ satisfies a suitable monotonicity property over $K$, then
a range of distributed projection-based
schemes~\cite{Facchinei_Pang03,alpcan03distributed,alpcan02game,pavel06noncooperative,pavel07extension}
and their regularized variants
schemes~\cite{yin09nash,yin09nash2,kannan10distributed,kannan10online}
can be constructed.  In all of these instances, every agent should be
able to observe the aggregate $\bar x$ of the agent decisions. In this
paper, we assume that this aggregate {\em cannot be observed} and {\em
no central entity exists} that can \us{globally broadcast} this quantity
at any time.  Yet, when agents are connected in some manner, then a
given agent may communicate locally with their neighbors and generate
estimates of the aggregate decisions.  Under this restriction, {\it we
	are interested in designing algorithms for computing an equilibrium
		of an aggregative Nash game}~\eqref{eqn:problem}.

\section{Distributed Synchronous Algorithm}\label{sec:synchron}
In this section we develop a distributed synchronous algorithm for
equilibrium computation of the game in~\eqref{eqn:problem} that
relies on agents constructing an estimate \us{of the aggregate} by {\em mixing} information
	drawn from local neighbors and making a subsequent projection step.
	In Section~\ref{sec:outline}, we describe the scheme and provide
	some preliminary results in Section~\ref{sec:prelim}.  This section
	concludes in Section~\ref{sec:convalgo} with an analysis of the
	convergence of the proposed scheme.

\subsection{Outline of Algorithm} \label{sec:outline}

Our algorithm equips each agent in the network with a protocol
that mandates that every agent exchange information with its
neighbors, and subsequently update its decision and the estimate
	of the aggregate decisions, simultaneously. We employ a synchronous
	time model which can contend with a time varying connectivity graph.
	Consequently, in this section we consider a time varying network to
	model agent's communications in time. More specifically, let
	$\Escr_k$ be the set of underlying undirected edges between agents
	and let $\Gscr_k=(\Nscr,\Escr_k)$ denote the connectivity graph at
	time $k.$ Let $\Nscr_i(k)$ denote the set of agents who are
	immediate neighbors of agent $i$ at time $k$ that can
	send information to $i$, assuming that $i \in \Nscr_i(k)$ for all
		$i\in \Nscr$ and all $k\ge0$. Mathematically,
$\Nscr_i(k)$ can be expressed as: 
\[\Nscr_i(k)= \{j \, : \{i,j\} \in \Escr_k\}.\]
 We make the following assumption on the graph $\Gscr_k=(\Nscr,\Escr_k)$.
\begin{assumption}\label{assump:connectivity}
 There exists an integer $Q \geq 1$ such that the graph $(\Nscr,\bigcup_{\ell=1}^{Q}
 \Escr_{\ell+k})$ is connected for all $k\ge0$.
\end{assumption}
This assumption ensures that the intercommunication intervals are bounded for
agents that communicate directly; i.e., every agent sends information to
each of its neighboring agents at least once every $Q$ time intervals.
This assumption has been commonly used in distributed algorithms on networks, 
starting with~\cite{Tsitsiklis84}.

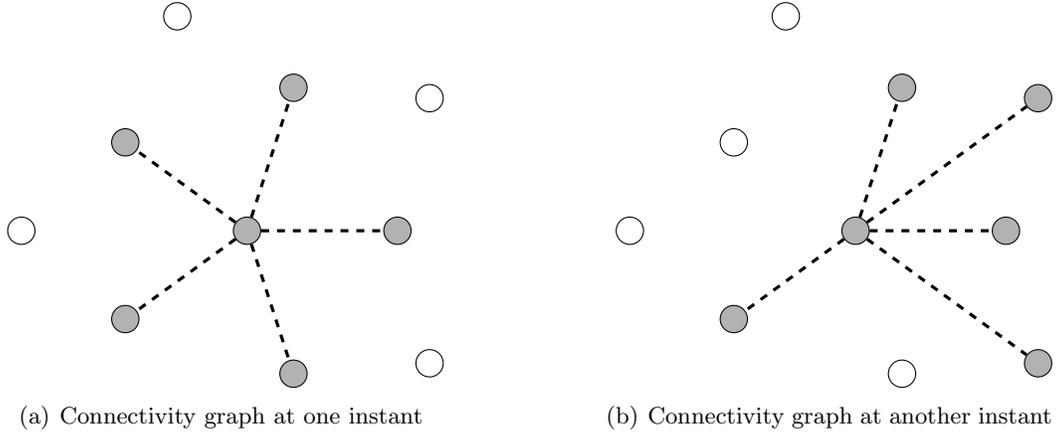
\begin{figure}[ht]
\centering
\subfigure[Connectivity graph at one instant]{
\begin{tikzpicture}[scale=2]
\tikzstyle{vertex}=[draw,shape=circle];
\tikzstyle{selected}=[vertex,fill=black];
\tikzstyle{neighbor}=[vertex,fill=black!30];
\tikzstyle{nonneighbor}=[vertex,fill=black!0];

\tikzstyle{edge} = [draw,thick,-]
\tikzstyle{selected edge} = [draw,line width=1.2pt,dashed,black]
\tikzstyle{ignored edge} = [draw,line width=1.2pt,dashed,black]

\path[selected] (0:0cm)    node[neighbor] (v0) {};
\path (0:1cm)    node[neighbor] (v1) {};
\path (72:1cm)   node[neighbor] (v2) {};
\path (2*72:1cm) node[neighbor] (v3) {};
\path (3*72:1cm) node[neighbor] (v4) {};
\path (4*72:1cm) node[neighbor] (v5) {};

\node[nonneighbor] (v6) at (36:1.5cm) {};
\node[nonneighbor] (v7) at (108:1.5cm) {};
\node[nonneighbor] (v8) at (180:1.5cm) {};
\node[nonneighbor] (v9) at (324:1.5cm) {};

\path[selected edge] (v0) -- (v1);
\path[ignored edge] (v0) -- (v2);
\path[ignored edge] (v0) -- (v3);
\path[ignored edge] (v0) -- (v4);
\path[ignored edge] (v0) -- (v5);
\end{tikzpicture}}
\hspace{1.9cm}
\subfigure[Connectivity graph at another instant]{
\begin{tikzpicture}[scale=2]
\tikzstyle{vertex}=[draw,shape=circle];
\tikzstyle{selected}=[vertex,fill=black];
\tikzstyle{neighbor}=[vertex,fill=black!30];
\tikzstyle{nonneighbor}=[vertex,fill=black!0];

\tikzstyle{edge} = [draw,thick,-]
\tikzstyle{selected edge} = [draw,line width=1.2pt,dashed,black]
\tikzstyle{ignored edge} = [draw,line width=1.2pt,dashed,black]

\path[selected] (0:0cm)    node[neighbor] (v0) {};
\path (0:1cm)    node[neighbor] (v1) {};
\path (72:1cm)   node[neighbor] (v2) {};
\path (2*72:1cm) node[nonneighbor] (v3) {};
\path (3*72:1cm) node[neighbor] (v4) {};
\path (4*72:1cm) node[nonneighbor] (v5) {};

\node[neighbor] (v6) at (36:1.5cm) {};
\node[nonneighbor] (v7) at (108:1.5cm) {};
\node[nonneighbor] (v8) at (180:1.5cm){};
\node[neighbor] (v9) at (324:1.5cm) {};

\path[selected edge] (v0) -- (v1);
\path[ignored edge] (v0) -- (v2);
\path[ignored edge] (v0) -- (v6);
\path[ignored edge] (v0) -- (v4);
\path[ignored edge] (v0) -- (v9);
\end{tikzpicture}}
\caption{A depiction of an undirected communication network.}
\label{fig:diffusion}
\end{figure}

Due to incomplete information at any point, an agent only has an estimate of
$\bar x$ in contrast to the actual $\bar x.$ We describe how an agent
may build this estimate. Let $x_i^k$ be the iterate and $v_i^k$ be the estimate of
the average of the decisions $x_1^k,\ldots,x_N^k$ for agent $i$ at the end of the $k$th iteration.
At the beginning of the $(k+1)$st iteration, agent $i$ receives the estimates $v_j^k$ from 
its neighbors $j\in \Nscr_i(k+1)$. Using this information, agent $i$
aligns its intermediate estimate according to the following rule: 
\begin{equation}
\label{eqn:estimate_mixing}
\hat v_i^k= \sum_{j \in \Nscr_i(k)} w_{ij}(k) v^k_j, 
\end{equation}
where $w_{ij}(k)$ is the nonnegative weight that agent $i$ assigns to agent
$j$'s estimate. By specifying $w_{ij}=0$ for $j \not \in \Nscr_i(k)$ we can write:
\[
\hat v_i^k= \sum_{j=1}^{N} w_{ij}(k) v_j^k \qquad \mbox{with} \qquad  v_j^0
= x_j^0 \mbox{ for all } j=1,\ldots,N.
\] 
Using this aligned average
estimate $\hat v_i^k$ and its own iterate $x_i^k$, agent $i$ updates its
iterate and average estimate as follows:
\begin{align}\label{eqn:x_algo}
 x_i^{k+1} & :=\Pi_{K_i}[x_i^k-\alpha_k F_i(x_i^k,N\hat v^k_i)], \\
\label{eqn:estimate_update}
 v_i^{k+1} & := {\hat v_i^k} + x_i^{k+1}-x_i^k,
\end{align}
where $\alpha_k$ is the stepsize, $\Pi_{K_i}\us{(u)}$ denotes the
Euclidean projection \us{of a vector $u$} onto the set $K_i$ and 
$F_i$ is as defined in~\eqref{eqn:compfi}. 
The quantity $N\hat v^k_i$ in~\eqref{eqn:x_algo} is the aggregate estimate that
agent $i$ uses instead of the true estimate $\sum_{i=1}^N x_i^k$ of the agent decisions at time $k$.
Under suitable conditions on the agents weights $w_{ij}(k)$ and the stepsize $\alpha_k$, 
the iterate vector $(x_1^k,\ldots,x_N^k)$ can converge to a Nash equilibrium point 
$(x_1^*,\ldots,x_N^*)$ and the estimates $N\hat v^k_i$ in~\eqref{eqn:x_algo} will converge
to the true aggregate value $\sum_{i=1}^N x_i^*$ at the equilibrium. These assumptions are 
given below.

\begin{assumption} \label{assump:weight} Let $W(k)$ be the weight matrix with entries $w_{ij}(k)$.
 For all $i \in \Nscr$ and all $k\ge0$, the following hold:
\begin{enumerate}[(i)]
\item $w_{ij}(k) \geq \delta$ for all $j\in \Nscr_i(k)$ and $w_{ij}(k)=0$ for $j \not \in \Nscr_i(k)$;
\item $\sum_{j=1}^{N} w_{ij}(k) =1$ for all $i$;
\item $\sum_{i=1}^{N} w_{ij}(k) =1$ for all $j.$
\end{enumerate}
\end{assumption}
Assumption~\ref{assump:weight} essentially requires every player to assign a positive  weight to the information received from its neighbor. Following Assumption~\ref{assump:weight} (ii)-(iii), the matrix $W(k)$ is doubly stochastic.
We point the reader to~\cite{nedic09distributed} for the
examples and a detailed discussion of the weights satisfying 
the preceding assumption.

\begin{assumption}\label{assump:step_diffusion}
The stepsize $\alpha_k$ is chosen such that the following hold:
\begin{enumerate}[(i)]
\item The sequence $\{\alpha_{k}\}$ is monotonically {non-increasing i.e., 
$\alpha_{k+1} \le \alpha_k$ for all $k$;}
\item $\sum_{{k=0} }^\infty \alpha_k = \infty$;
\item $\sum_{{k=0} }^\infty \alpha_k^2 < \infty.$
\end{enumerate}
\end{assumption}
Such an assumption is satisfied for a stepsize of the form
{$\alpha_{k} = (k+1)^{-b}$} where $0.5 < b \leq 1$.

\subsection{Preliminary Results}\label{sec:prelim}
We next provide some auxiliary results for the weight matrices and the estimates generated by the method. We introduce
the transition matrices $\Phi(k,s)$ from time $s$ to $k>s$,
	as follows:
\[\Phi(k,s)=W(k)W(k-1)\cdots W(s+1)W(s)\qquad\hbox{for } 0\le s<k,\] 
where $\Phi(k,k)={W(k)}$ for all $k.$ Let $[\Phi(k,s)]_{ij}$ denote
the $(i,j)$th entry of the matrix $\Phi(k,s),$ and let $\mathbf{1} \in
\mathbb{R}^{N}$ be the column vector with all entries equal to 1.  We
next state a result on the convergence properties of the matrix
$\Phi(k,s).$ The result can be found in \cite{aa_dist_avg_2008}
(Corollary 1).

\begin{lemma}[\cite{Tsitsiklis84} Lemma 5.3.1]\label{lemma:weight}
 Let {Assumptions}~\ref{assump:connectivity} and \ref{assump:weight}
 hold. Then \us{the following hold:}
\begin{enumerate}[(i)]
 \item $\lim_{k \to \infty} \Phi(k,s) = \frac{1}{N} \mathbf{1}\mathbf{1}^T$ for all $s\ge0.$
 \item The convergence rate of $\Phi(k,s)$ is geometric; specifically, we have
$\left|[\Phi(k,s)]_{ij}-\frac{1}{N}\right| \leq \theta \beta^{k-s}$ for all $k\ge s\ge0$ 
and for all $i$ and $j$,
where $\theta=(1-\frac{\delta}{4N^2})^{-2}$ and $\beta=(1-\frac{\delta}{4N^2})^{\frac{1}{Q}}.$
\end{enumerate}
\end{lemma}
Next, we state some results which will allow us to claim the convergence of the algorithm.
These results involve the average  of the estimates $v^k_i,i\in\Nscr$,
	  \us{defined by $y^k$:} 
\begin{equation}\label{eqn:true_agg1}
 y^k = {\frac{1}{N} \sum_{i=1}^N v_i^k} \quad \mbox{for all } k \geq 0.
\end{equation}
As we \us{proceed to show}, $y^k$ will play a key role in establishing
the convergence of the iterates produced by the algorithm
in~\eqref{eqn:x_algo}--\eqref{eqn:estimate_update}.  One important
property of $y^k$ is that we have $y^k= \frac{1}{N} \sum_{j=1}^N x^k_j$
for all $k \geq 0.$ Thus, $y^k$ not only captures the average belief of
the agents in the network but it also represents the true average
information.  This property of the true average $y^k$ has been shown
in~\cite{ram2012} within the proof of Lemma 5.2 for a different setting,
	and it is given in the following lemma for sake of clarity.

\begin{lemma}\label{lemma:true_agg} 
Let $W(k)$ be such that  $\sum_{j=1}^N [W(k)]_{ji}=1$ for every $i$ and
$k$.  Then,  $y^k = \frac{1}{N} \sum_{i=1}^N x^k_i$ for all $k\ge0$,
	where $y^k$ is defined by~\eqref{eqn:true_agg1}.
\end{lemma}
\begin{proof}
It suffices to show that for all $k\ge0$,
\begin{align}\label{eqn:relo1}
\sum_{j=1}^{N} v^k_j = \sum_{j=1}^{N} x^k_j. 
\end{align}
We show this by induction on $k$. 
For $k=0$ relation \eqref{eqn:relo1} holds trivially, as we have initialized the beliefs with
$v_j^0 = x_j^0$ for all $j$. 
Assuming relation \eqref{eqn:relo1} holds for $k-1,$ as the induction step, we have
\begin{align*}
\sum_{j=1}^{N} {v_j^k} = & \sum_{j=1}^{N} (\hat v_j^{k-1}+x_j^{k}-x_j^{k-1}) \\
     = & \sum_{j=1}^N \sum_{i=1}^N [W(k-1)]_{ji} v_i^{k-1} +\sum_{j=1}^{N} (x_j^{k}-x_j^{k-1}) \\
     =& \sum_{i=1}^N v_i^{k-1}+\sum_{j=1}^{N} (x_j^{k}-x_j^{k-1}),
\end{align*}
where the first equality follows from \eqref{eqn:estimate_update},
	the second inequality is a consequence of the mixing relationship
		articulated by \eqref{eqn:estimate_mixing}, and  the 
	last equality follows from $\sum_{j=1}^N [W(k)]_{ji}=1$ for every $i$ and $k$. 
Furthermore, using the induction hypothesis, we have 
$\sum_{j=1}^{N} (x_j^{k}-x_j^{k-1})= \sum_{j=1}^N x_j^{k} -\sum_{j=1}^{N} v^{k-1}_j,$
thus implying that
$\sum_{j=1}^{N} v^k_j =\sum_{j=1}^N x_j^{k}.$
\end{proof}

As a consequence of Lemma~\ref{lemma:true_agg},
   Assumptions~\ref{assump:set_func} and ~\ref{assump:Lipschitz}, we have the following result which will be often used in the
sequel.
\begin{lemma}\label{cor:fibound} 
\us{Let $W(k)$ be such that  $\sum_{j=1}^N [W(k)]_{ji}=1$ for every $i$ and
$k$.} 
Also, let Assumptions~\ref{assump:set_func} and~\ref{assump:Lipschitz} hold.
Then, there exists a constant $C$ such that 
\[\| F_i(x_i^{k}, Ny^{k})\|\le C, \qquad \| F_i(x_i^k, N\hat v_i^{k})\| \le C
\qquad\hbox{for all $i$ and $k\ge0$.}\]
\end{lemma}
\begin{proof}
By Lemma~\ref{lemma:true_agg}, we have $Ny^{k}=\sum_{i=1}^N x_i^k=\bar x_k\in \bar K$, where $\bar K$ 
is compact since each $K_i$ \us{is} compact (Assumption~\ref{assump:set_func}). Since each $F_i$ is continuous 
over $K_i\times \bar K$, the first inequality follows. 
To show that $\{F_i(x_i^k, N\hat v_i^{k})\}$ is bounded,
we write 
\[\| F_i(x_i^{k}, N\hat v_i^{k})\|\le \| F_i(x_i^{k}, N\hat v_i^{k})-F_i(x_i^{k}, Ny^{k})\|
+ \| F_i(x_i^{k}, Ny^{k})\|.\]
Using the Lipschitz property of $F_i$ of Assumption~\ref{assump:Lipschitz},
we obtain
\[\| F_i(x_i^{k}, N\hat v_i^{k})\|\le \bar L_i N\|\hat v_i^{k} - y^{k}\|+ \| F_i(x_i^{k}, Ny^{k})\|.\]
Let $\hat K$ be the convex hull of the union set $\cup_{i}K_i$. Note that $\hat v^k_i,y^k\in \hat K$ 
for all $k$ and that $\hat K$ is compact (since each $K_i$ is compact).\footnote{Though $y^k \in \bar K,$ we cannot claim the same for $\hat v^k_i$.}
Thus, $\{\|\hat v_i^{k} - y^{k}\|\}$ is bounded. As already established, $\{F_i(x_i^{k}, Ny^{k})\}$
is also bounded, implying that $\{F_i(x_i^k, N\hat v_i^{k})\}$ is bounded as well.
\end{proof}

In the following lemma, we establish \us{an error bound on} the norm
$\|y^k-\hat v^k_i\|$ which \us{plays an} important role in our analysis.

\begin{lemma}\label{lemma:yk_vk_diffusion}
Let Assumptions~\ref{assump:set_func}--\ref{assump:weight}
hold, and let $y^k$ be defined by~\eqref{eqn:true_agg1}. Then, we have 
\[
\|y^k - \hat v_i^k\|  \leq \theta \beta^{k}M 
+\theta NC\sum_{s=1}^{k} \beta^{k-s}\alpha_{s-1}
\qquad \mbox{for all $i\in\Nscr$ and all $k\ge1$}, 
\]
where  $\hat v^k_i$ is defined in~\eqref{eqn:estimate_mixing},
$\theta=(1-\frac{\delta}{4N^2})^{-2}$, $\beta=(1-\frac{\delta}{4N^2})^{\frac{1}{Q}}$,
$M=\sum_{j=1}^N\max_{x_j\in K_j} \|x_j\|$ and $C$ \us{denotes the bound} in \us{Lemma}~\ref{cor:fibound}.
\end{lemma}
\begin{proof}
Using the definitions of $v^{k+1}_i$ and $\hat v_i^k$ given in 
Eqs.~\eqref{eqn:estimate_update} and~\eqref{eqn:estimate_mixing}, respectively, we have
\begin{align*}
v_i^{k+1} & =\sum_{j=1}^{N} w_{ij}(k)v_j^k +x_i^{k+1}-x_i^k,
\end{align*}
which through an iterative recursion leads to
\begin{align*}
v_i^{k+1}
& =  \sum_{j=1}^{N} w_{ij}(k)\left(  
\sum_{\ell=1}^{N} w_{j\ell}(k-1) v_\ell^{k-1}+  x_j^{k}-x_j^{k-1} \right)
+x_i^{k+1}-x_i^k \\
& = \sum_{\ell=1}^{N} [\Phi(k,k-1)]_{i\ell} v_\ell^{k-1} 
+\sum_{j=1}^{N} [\Phi(k,k)]_{ij}\left( x_j^{k} - x_j^{k-1} \right)+ x_i^{k+1}-x_i^k \\
&=\cdots\\
& = \sum_{\ell=1}^{N} [\Phi(k,0)]_{i\ell}v_\ell^0
+\sum_{s=1}^{k} \left( \sum_{j=1}^{N} [\Phi(k,s)]_{ij}(x_j^s-x_j^{s-1}) \right) + x_i^{k+1}-x_i^k.
\end{align*} 
The preceding relation can be rewritten as:
\begin{equation*}
v_i^{k+1}-x_i^{k+1}+x_i^k = \sum_{\ell=1}^{N} [\Phi(k,0)]_{i\ell}v_\ell^0 
+ \sum_{s=1}^{k} \left( \sum_{j=1}^{N} [\Phi(k,s)]_{ij}(x_j^s-x_j^{s-1}) \right).
\end{equation*}
By the definition of $v^{k+1}_i$ in Eq.~\eqref{eqn:estimate_update}, we have
$\hat v^k_i=v_i^{k+1} -x_i^{k+1}+x_i^k$, 
through which we get
\begin{equation}\label{eqn:x_estimate}
\hat v^k_i = \sum_{\ell=1}^{N} [\Phi(k,0)]_{i\ell}v_\ell^0
+\sum_{s=1}^{k} \left( \sum_{j=1}^{N} [\Phi(k,s)]_{ij}(x_j^s-x_j^{s-1}) \right).
\end{equation}
Now, consider $y^k$ which may be written as follows:
\[y^k=y^{k-1} + (y^k-y^{k-1}) =\cdots=y^0+\sum_{s=1}^k (y^s-y^{s-1}).\]
By Lemma~\ref{lemma:true_agg} we have 
$y^s = \frac{1}{N} \sum_{j=1}^N x^s_j$ for all $s\ge0$, which implies
\begin{align}\label{eqn:y_rel}
y^k=y^0 + \sum_{s=1}^k \sum_{j=1}^N \frac{1}{N}\left( x^s_j - x^{s-1}_j\right)
=\sum_{\ell=1}^{N} \frac{1}{N}v_\ell^0
+\sum_{s=1}^k \sum_{j=1}^N \frac{1}{N}\left( x^s_j - x^{s-1}_j\right),
\end{align}
where the last equality follows by the definition of $y^0$ (see~\eqref{eqn:true_agg1}).

From relations~\eqref{eqn:x_estimate} and \eqref{eqn:y_rel} we have
\begin{align}\label{eqn:ykvk_rel}
\|y^k-\hat v_i^k\|
  & =  \left\|\sum_{\ell=1}^{N} \left( \frac{1}{N} - [\Phi(k,0)]_{i\ell}\right) v_\ell^0
+  \sum_{s=1}^k \sum_{j=1}^N \left(\frac{1}{N} - [\Phi(k,s)]_{ij}\right)\left( x^s_j - x^{s-1}_j\right) 
  \right\| \nonumber\\
  & \leq \sum_{\ell=1}^{N} \left|\frac{1}{N} - [\Phi(k,0)]_{i\ell}\right| \left\|v_\ell^{0}\right\|
  +\sum_{s=1}^{k}\sum_{j=1}^{N} 
  \left|\frac{1}{N}- [\Phi(k,s)]_{ij}\right|\left\| x_j^s-x_j^{s-1} \right\| \cr
&\leq \sum_{\ell=1}^{N} \beta^k \left\|v_\ell^{0}\right\|
  +\sum_{s=1}^{k}\sum_{j=1}^{N} \beta^{k-s}\left\| x_j^s-x_j^{s-1} \right\|
\end{align}
where the last inequality follows from 
$\left| \frac{1}{N}-[\Phi(k,s)]_{ij}\right|\leq \theta \beta^{k-s}$ for all $0\le s\le k$ 
(cf.~Lemma~\ref{lemma:weight}).

Now, we estimate $\|x_i^{s}-x_i^{s-1}\|.$  
From relation~\eqref{eqn:x_algo} we see that for any $s \geq 1$,
\begin{align}\label{eqn:boundxk}
\|x_i^{s}-x_i^{s-1} \|
& = \|\Pi_{K_i}[x_i^{s-1}-\alpha_{s-1} F_i(x_i^{s-1},N\hat v_i^{s-1})]-x_i^{s-1} \| \nonumber \\
& \leq \|x_i^{s-1}-\alpha_{s-1} F_i(x_i^{s-1},N\hat v_i^{s-1})-x_i^{s-1} \| \nonumber \\
& =\alpha_{s-1}\| F_i(x_i^{s-1}, N\hat v_i^{s-1})\|\cr
& \le C\alpha_{s-1},
\end{align}
where the first inequality follows by the non-expansive property of projection map,
and the \us{second} inequality follows by Lemma~\ref{cor:fibound}.
Combining~\eqref{eqn:boundxk} and~\eqref{eqn:ykvk_rel}, 
we have
\begin{align*}
 \|y^k-\hat v_i^k\| 
 \leq  \theta \beta^{k}\sum_{\ell=1}^{N}\|v_\ell^0\|
+ \theta N\sum_{s=1}^{k} \beta^{k-s}\alpha_{s-1} C 
\le \theta \beta^{k}M +\theta NC\sum_{s=1}^{k} \beta^{k-s}\alpha_{s-1},
\end{align*}
where in the last inequality, we use
$v_\ell^0=x_\ell^0\in K_\ell$ and $M=\sum_{\ell=1}^{N}\max_{x_\ell\in K_\ell}\|x_\ell^0\|,$
which is finite since each $K_\ell$ is a compact set (cf.~Assumption~\ref{assump:set_func}).
\end{proof}
From the right hand side of the expression in
Lemma~\ref{lemma:yk_vk_diffusion}, it is apparent that the parameter for
network connectivity, $Q$ (cf.~Assumption~\ref{assump:connectivity})
	determines the rate of convergence of player's estimate of the
	aggregate to the actual aggregate. If the network connectivity is
	\us{poor}, 
$Q$ is large implying $\beta$ is close to 1 resulting in a \us{slower
convergence rate.} 

\subsection{Convergence \us{theory}}\label{sec:convalgo}

In this subsection, under our assumptions, we prove that the sequence produced by the
	 proposed algorithm does indeed converge to the unique Nash
		 equilibrium, which exists by Proposition~\ref{uniq-equil}.
		 Our next proposition provides the main convergence result for the
		 algorithm. Prior to providing this result, we
		 state two lemmas that will be employed in proving the required
		 result, the first being a supermartingale convergence result (see
		 for example \cite[Lemma 11, Pg. 50]{Polyakbook87})
	 and the second being \cite[Lemma 3.1(b)]{ram_2010_stocsub}. 

\begin{lemma}\label{lemma:supermartingale}
Let $V_k, u_k,\beta_k$ and $\gamma_k$ be non-negative random variables adapted
to some $\sigma$-algebra $\mathcal{F}_k$. If almost surely
$\sum_{k=0}^{\infty}u_k < \infty$, $\sum_{k=0}^{\infty}\beta_k<\infty$,
and
\[ \mathbb{E}[V_{k+1} \mid\mathcal{F}_k] \leq  (1 + u_k)V_k-\gamma_k+\beta_k
\quad\hbox{for all }k\ge0,\]
then almost surely {$V_k$ converges} and
$\sum_{k=0}^{\infty}\gamma_k< \infty$.
\end{lemma}

\begin{lemma}\cite[Lemma 3.1(b)]{ram_2010_stocsub}
\label{lem:sum_conv}
Let $\{\zeta_k\}$ be a non-negative scalar sequence. If $\sum_{k=0}^\infty \zeta_k <  \infty$ 
and $0 < \beta < 1,$ then 
$\sum_{k=0}^\infty \left(\sum_{s=0}^k \beta^{k-s}\zeta_{s}\right) < \infty.$
\end{lemma}

In what follows, we use $x^k$ to denote the vector with components $x_i^k$, $i=1,\ldots,N$,
i.e., $x^k=(x_1^k,\ldots,x_N^k)$ and, similarly, we write $x^*$ for the vector $(x_1^*,\ldots,x^*_N)$.

\begin{proposition}
\label{prop:diffusion_convergence}
Let Assumptions~\ref{assump:set_func}--\ref{assump:step_diffusion} hold.
Then, the sequence $\{x^k\}$ generated by the method
\eqref{eqn:x_algo}--\eqref{eqn:estimate_update} 
converges to the  (unique) solution $x^*$ of VI($K,\phi)$. 
\end{proposition}
\begin{proof}
By Proposition~\ref{uniq-equil}, VI$(K,\phi)$ has a unique solution $x^*\in K$.
When $x^*$ solves the variational inequality problem VI$(K,\phi)$, 
the following relation holds 
$x^*=\Pi_{K_i}[x^*_i-\alpha_k F_i(x^*_i,\bar{x}^*)]$ 
(see \cite[Proposition 1.5.8, p.~83]{Facchinei_Pang03}).
From this relation and the non-expansive property of projection
operator, we see that
\begin{align*}
\|x_i^{k+1}-x_i^*\|^2 
 & = \|\Pi_{K_i}[x_i^k-\alpha_k F_i(x_i^k,N{\hat v_i^k})] - x^*_i\|^2\cr
 & = \|\Pi_{K_i}[x_i^k-\alpha_k F_i(x_i^k,N{\hat v_i^k})] 
         - \Pi_{K_i}[x^*_i-\alpha_k F_i(x^*_i,\bar{x}^*)]\|^2\cr
 & \leq \|x_i^k-x^*_i-\alpha_k (F_i(x_i^k,N{\hat v_i^k}) - F_i(x^*_i,\bar{x}^*))\|^2.
\end{align*}
By expanding the last term, we obtain the following expression:
\begin{align}\label{rel1}
\|x_i^{k+1}-x_i^*\|^2  
\leq &\|x_i^k-x^*_i\|^2
+ \alpha_k^2 \underbrace{ \|F_i(x_i^k,N{\hat v_i^k}) -
	F_i(x^*_i,\bar{x}^*)\|^2}_{\mbox{\bf Term 1}}\cr
&
- 2 \alpha_k \underbrace{ (F_i(x_i^k,N{\hat v_i^k}) -
		F_i(x^*_i,\bar{x}^*))^T(x_i^k-x^*_i)}_{\mbox{\bf Term 2}}.
\end{align}
To estimate {\bf Term 1}, we use the triangle inequality and the identity 
$(a+b)^2 \leq 2(a^2+b^2)$, which yields 
\begin{align*}
 \mbox{\bf Term 1}
\le 2\|F_i(x_i^k,N\hat v_i^k)\|^2 + 2 \|F_i(x^*_i,\bar{x}^*)\|^2
\le \tilde C\qquad\hbox{with}\quad
 \tilde C= 2C^2+2\max_{(x_i,\bar x)\in K_i\times \bar K}\|F_i(x_i,\bar{x})\|^2,
\end{align*}
where $C$ is such that $\|F_i(x_i^k,N\hat v_i^k)\|\le C$ for all $i$ and
$k$ (cf. Lemma~\ref{cor:fibound})
and $\max_{(x_i,\bar x)\in K_i\times \bar K}\|F_i(x_i,\bar x)\|$ is finite by 
Assumption~\ref{assump:set_func}. 
Next, we consider $\mbox{\bf Term 2}$.
By adding and subtracting $F_i(x_i^k,Ny^k)$ in $\mbox{\bf Term 2}$, where $y^k$ is defined 
by~\eqref{eqn:true_agg1}, we have
\begin{align*}
 \mbox{\bf Term 2}
=  \left (F_i(x_i^k,N \hat v_i^k) - F_i(x_i^k,Ny^k)\right)^T(x_i^k-x^*_i) 
+\left(F_i(x^k_i,Ny^k) - F_i(x^*_i,\bar{x}^*)\right)^T(x_i^k-x^*_i).
\end{align*}
By applying the \jk{Cauchy-Schwarz} inequality, i.e. $a^Tb \ge- \|a\|\|b\|$, 
to the first term on the right hand side of 
the preceding relation and the Lipschitz continuity of $F_i(x_i,u)$ in $u$ 
(cf.~Assumption~\ref{assump:Lipschitz}), we see that
\begin{align*}
(F_i(x_i^k,N\hat v_i^k) - F_i(x_i^k,Ny^k))^T(x_i^k-x^*_i)
& \ge- \|F_i(x_i^k,N\hat v_i^k)-F_i(x_i^k,Ny^k)\|\cdot \|x^k_i-x^*_i\|\cr
& \ge -\bar L_iN\|\hat v_i^k -y^k\|\cdot\|x^k_i-x^*_i\|\cr
&\ge -2\bar L_i MN\|\hat v_i^k -y^k\|,
\end{align*}
where in the last inequality we use $x_i^k,x_i^*\in K_i$ and the compactness of $K_i$ 
(cf.~Assumption~\ref{assump:set_func}) and $M\ge \max_{x_i\in K_i}\|x_i\|$ for all $i$.
Therefore, we have 
\begin{align*}
 \mbox{\bf Term 2}
\ge -2\bar L_i MN\|\hat v_i^k -y^k\|+\left(F_i(x^k_i,Ny^k) - F_i(x^*_i,\bar{x}^*)\right)^T(x_i^k-x^*_i).
 \end{align*}
By substituting the preceding estimates of 
$\mbox{\bf Term 1}$ and $\mbox{\bf Term 2}$ in~\eqref{rel1}, we obtain
\begin{align*}
\|x_i^{k+1}-x_i^*\|^2 
&\leq  \|x_i^k-x^*_i\|^2 + \tilde C\alpha_k^2
      +4\alpha_k \bar L_i M N\|\hat v_i^k -y^k\|\cr
& \;- 2 \alpha_k\left(F_i(x^k_i,Ny^k) - F_i(x^*_i,\bar{x}^*)\right)^T(x_i^k-x^*_i).
\end{align*}
Summing over all agents from $i=1$ to $i=N$, yields
\begin{align*}
\sum_{i=1}^N\|x^{k+1}_i-x^*_i\|^2 
& \leq  \sum_{i=1}^N \|x^k_i-x^*_i\|^2 +N\tilde C\alpha_k^2 
+4\alpha_k MN\sum_{i=1}^N \bar L_i \|\hat v_i^k - y^k\| \cr
& \; - 2 \alpha_k\sum_{i=1}^N\left(F_i(x^k_i,Ny^k) - F_i(x^*_i,\bar{x}^*)\right)^T(x_i^k-x^*_i).
\end{align*}
Using $Ny^k =\sum_{i=1}^N x^k_i$ (see Lemma~\ref{lemma:true_agg}) and letting 
$\bar x^k= \sum_{i=1}^N x^k_i$,  we have for all $k\ge0$,
\begin{align}\label{eqn:last}
\|x^{k+1}-x^*\|^2 
 & \leq  \|x^k-x^*\|^2 + N\tilde C\alpha_k^2
+4\alpha_kMN\sum_{i=1}^N \bar L_i \|\hat v_i^k - y^k\| \cr
&\;- 2 \alpha_k \left(\phi(x^k) - \phi(x^*)\right)^T(x^k-x^*),
\end{align}
where we also use the fact that $F_i(x_i,\bar x)$ is a coordinate map for the mapping 
$\phi(x)=F(x,\bar x)$ (see~\eqref{eqn:svinot} and~\eqref{eqn:phimap}).  
To claim \us{that the sequence}  $\{x^k\}$ converges to $x^*$, 
we apply Lemma~\ref{lemma:supermartingale} 
(for the deterministic sequences) to relation~\eqref{eqn:last}. 
To apply this lemma, since $\sum_{k=0}^\infty\alpha_k^2<\infty$ 
by Assumption~\ref{assump:step_diffusion}, we only need to prove
\begin{align}\label{eqn:key1}
\sum_{k=0}^\infty \alpha_k\|\hat v_i^k - y^k\|<\infty\qquad\hbox{for all }i\in\Nscr.\end{align}
In view of Lemma~\ref{lemma:yk_vk_diffusion}, we have
\[\|y^k - \hat v_i^k\|  \leq  \theta \beta^{k}M+ \theta NC\sum_{s=1}^{k} \beta^{k-s}\alpha_{s-1} 
\qquad\hbox{for all $i\in\Nscr$ and all $k\ge 1$},\]
so it suffices to prove that
\[\sum_{k=1}^\infty \alpha_k\left(\sum_{s=1}^{k} \beta^{k-s}\alpha_{s-1}\right)<\infty
\qquad\hbox{and}\qquad
\sum_{k=1}^\infty \alpha_k\beta^{k}<\infty.\]
Using $\alpha_k \leq \alpha_{s}$ for all $k \ge s$ (Assumption~\ref{assump:step_diffusion}),
for the series $\sum_{k=1}^\infty \alpha_k\left(\sum_{s=1}^{k} \beta^{k-s}\alpha_{s-1}\right)$
we have
\begin{align*}
\sum_{k=1}^\infty \alpha_k\left(\sum_{s=1}^{k} \beta^{k-s}\alpha_{s-1}\right) = 
\sum_{k=1}^\infty \left(\sum_{s=1}^{k} \beta^{k-s}\alpha_k\alpha_{s-1}\right) 
 \leq \sum_{k=1}^\infty \left(\sum_{s=1}^{k} \beta^{k-s}\alpha_{s-1}^2\right).
\end{align*}
We now use Lemma \ref{lem:sum_conv}, from which 
by letting $\zeta_{s} = \alpha_{s}^2$ we can see that
$\sum_{k=1}^\infty \alpha_k\left(\sum_{s=1}^{k} \beta^{k-s}\alpha_{s-1}\right) < \infty.$
To establish the convergence of $\sum_{k=0}^\infty\alpha_k\beta^{k}$, we note that 
$\alpha_k\le \alpha_0$  
(Assumption~\ref{assump:step_diffusion}), implying that 
$\sum_{k=0}^\infty \alpha_k\beta^k\le
\alpha_0\sum_{k=0}^\infty\beta^k<\infty$ since $0<\beta<1$. 
Thus, relation~\eqref{eqn:key1} is valid.

As relation~\eqref{eqn:last} satisfies the conditions of (the deterministic case of) 
Lemma~\ref{lemma:supermartingale}, 
it follows that
\begin{equation}\label{eqn:seq_conv_det}
\{\|x^k-x^*\|\} \quad \mbox{is a convergent sequence,}
\end{equation}
\begin{equation}\label{eqn:F_conv_det}
\mbox{ and }  \sum_{k=0}^\infty \alpha_k (\phi(x^k) - \phi(x^*) )^T(x^k-x^*) < \infty.
\end{equation}
Since $\{x^k\}\subset K$ and $K$ is compact (Assumption~\ref{assump:set_func}),
$\{x^k\}$ has accumulation points in $K$.
By~\eqref{eqn:F_conv_det} and $\sum_{k=0}^\infty \alpha_k = \infty$
it follows that $(\phi(x^{k}) - \phi(x^*))^T (x^k-x^*)\to0$
along a subsequence, say $\{x^{k_\ell}\}$. This observation,
	together with the strict monotonicity of $\phi$,
implies that $\{x^{k_\ell}\}\to x^*$ as $\ell\to\infty$. \uvs{To claim
	$\|x^k-x^*\| \to 0,$  {we proceed by contradiction and} assume that
		$\|x_k-x^*\|$ does not converge to $0$. Then there exists an $\epsilon > 0$
and a subsequence $\{x_{n_l}\}$ such that $\|x_{n_l} - x^*\| \geq \epsilon$
for $l$ sufficiently large. But
since the set is bounded, there exists a further subsequence which
converges to some $\tilde{x}$. By relation~\eqref{eqn:seq_conv_det}, the
sequence $\{x_k\}$ is convergent and therefore, it
must be that $\tilde{x} = x^*$ \uvs{ and } 
the entire sequence $\{x^k\}$ must converge to $x^*$.}
\end{proof}

Though it is difficult to make a statement 
on the rate of convergence for the result of  Proposition~\ref{prop:diffusion_convergence}, 
the network connectivity plays an important 
role in determining the rate as already discussed for the results of
Lemma~\ref{lemma:yk_vk_diffusion}. Indeed, if $\beta$ is close to 1,
	which is the case for a network with \us{poor} connectivity, players take longer to converge on their estimate of the true aggregate, thereby taking it longer to converge on their optimal decision.

\section{Distributed Asynchronous Algorithm}\label{sec:asynchron}
In this section, we propose a distributed gossip-based
	algorithm for computing an equilibrium of aggregative Nash
game, \us{as defined by}~\eqref{eqn:problem}. In a gossip protocol, the
information is propagated by running a \us{round of} information
exchange between a randomly chosen player who communicates with another
player chosen at random. A more detailed description of the algorithm
and some preliminary results are provided in section~\ref{sec:4.1}. The
global convergence of the algorithm is examined in
section~\ref{sec:convergence}, while constant steplength error bounds
are provided in section~\ref{sec:4.3}.

\subsection{Outline of Algorithm}\label{sec:4.1}
In the proposed algorithm, agents perform their estimate and iterate
updates the same as in the synchronous
algorithm~\eqref{eqn:x_algo}--\eqref{eqn:estimate_update}, but the
updates occur asynchronously. As a mechanism for generating asynchronous
updates we employ the gossip model for agent
communications~\cite{boyd_gossip_2006}. Together with the asynchronous
updates, we allow the agents to use uncoordinated stepsize values by
letting each agent choose a stepsize based on its own information-update
frequency. To accommodate these updates and 
stepsize selections, we model the agent connectivity structure by an
undirected static graph $\Gscr(\Nscr,\Escr)$, with node $i\in\Nscr$
being agent $i$ and $\Escr$ being the set of undirected edges among the
agents. When $\{i,j\}\in\Escr$, the agents $i$ and $j$ \us{may} talk to
each other.  We let $\Nscr_i$ denote the set of neighbors of agent $i,$
i.e., $\Nscr_i= \{j \mid \{i,j\} \in \Escr\}.$ We use the following
assumption for the graph $\Gscr(\Nscr,\Escr)$.
\begin{assumption}\label{asum:graph}
The undirected graph $\Gscr(\Nscr,\Escr)$ is connected.
\end{assumption}

We use a gossip protocol to model agent communication and exchange of
the estimates of the aggregate $\bar x$.  In this model, each agent is
assumed to have a local clock which ticks according to a Poisson process
with rate~1. At a tick of its clock, an agent $i$ wakes up and contacts
its neighbor $j\in\Nscr_i$ with probability $p_{ij}$.  The agents'
clocks processes can be equivalently modeled as a single (virtual) clock
which ticks according to a Poisson process with rate~$N$. We assume that
only one agent wakes up at each tick of the global clock, and we let
$Z^k$ denote $k$th tick time of the global Poisson process.  We
discretize time so that instant $k$ corresponds to the time-slot
$[Z^{k-1},Z^k)$.  At each time $k$, every agent $i$ has its iterate
$x_i^k$ and estimate $v^{k}_i$ of the average of the current aggregate.
We let $I^k$ denote the agent whose clock ticked at time $k$ and we let
$J^k$ be the agent contacted by the agent $I^k$, where $J^k$ is a
neighbor of agent $I^k,$ i.e., $J^k \in \mathcal{N}_{I^k}$.
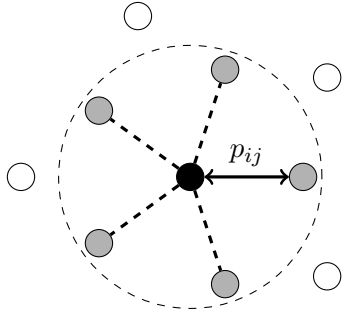
\begin{figure}
\centering
\begin{tikzpicture}[scale=1.5]
\tikzstyle{vertex}=[draw,shape=circle];
\tikzstyle{selected}=[vertex,fill=black];
\tikzstyle{neighbor}=[vertex,fill=black!30];
\tikzstyle{nonneighbor}=[vertex,fill=black!0];

\node[circle,color=black,minimum width=3.5cm,draw,dashed] (v) {};
\tikzstyle{edge} = [draw,thick,-]
\tikzstyle{selected edge} = [draw,line width=1.2pt,<->,black]
\tikzstyle{ignored edge} = [draw,line width=1.2pt,dashed,black]

\path[selected] (0:0cm)    node[selected] (v0) {};
\path (0:1cm)    node[neighbor] (v1) {};
\path (72:1cm)   node[neighbor] (v2) {};
\path (2*72:1cm) node[neighbor] (v3) {};
\path (3*72:1cm) node[neighbor] (v4) {};
\path (4*72:1cm) node[neighbor] (v5) {};

\node[nonneighbor] (v6) at (36:1.5cm) {};
\node[nonneighbor] (v7) at (108:1.5cm) {};
\node[nonneighbor] (v8) at (180:1.5cm) {};
\node[nonneighbor] (v9) at (324:1.5cm) {};

\path[selected edge] (v0) edge node[above] {$p_{ij}$} (v1);
\path[ignored edge] (v0) -- (v2);
\path[ignored edge] (v0) -- (v3);
\path[ignored edge] (v0) -- (v4);
\path[ignored edge] (v0) -- (v5);
\end{tikzpicture}
\caption{A depiction of a gossip communication.}
\end{figure}
At time $k$, agents $I^k$ and $J^k$ exchange their estimates $v^k_{I^{k}}$ and 
$v^k_{J^{k}}$ and compute intermediate estimates:
\begin{equation}\label{eqn:gossip_mixing}
 \hat v_i^k =  \frac{ v^k_{I^{k}} + v^k_{J^{k}}}{2} \quad \mbox{for } i \in \{I^k,J^k\},
\end{equation}
and update their iterates and estimates of the aggregate average, as follows:
\begin{align}\label{eqn:gossip_up}
\left. \begin{array}{l}
x_i^{k+1} :=\Pi_{K_i}[x_i^k-\alpha_{k,i} F_i(x_i^k,N\hat v^k_i)]\cr
v_i^{k+1}  :=\hat v_i^k+x_i^{k+1}-x_i^k 
\end{array}\right\}\qquad \mbox{for } i \in \{I^k,J^k\},
\end{align}
where $\alpha_{k,i}$ is the stepsize for agent $i$ and $F_i(x_i,y)=\nabla_{x_i}f_i(x_i,y).$
The other agents do nothing, i.e.,
\begin{align}\label{eqn:gossip_noup}
\hat v_i^k = v^k_i,\quad x_i^{k+1}=x_i^{k}, \quad\hbox{and}\quad  v_i^{k+1}=v_i^{k}
\qquad\hbox{for $i \not \in \{I^k,J^k\}$}.
\end{align}
As seen from the preceding update relations, the agents perform the same updates as in
the synchronous algorithm~\eqref{eqn:x_algo}--\eqref{eqn:estimate_update},
but instead of all agents updating, only two randomly selected agents update their estimates
and iterates, while the other agents do not update.

We now rewrite the update steps more compactly.
To capture the step in~\eqref{eqn:gossip_mixing}, we define the weight matrix $W(k)$:
\begin{equation}
\label{eqn:weight_matrix}
W(k) =  \mathbb{I} - \frac{1}{2}(e_{I^{k}}-e_{J^{k}})(e_{I^{k}}-e_{J^{k}})^T,
\end{equation}
where $\mathbb{I}$ stands for the identity matrix, 
$e_i$ is $N$-dimensional vector with $i$th entry equal to 1, and the other entries equal to 0.
By using $W(k)$ we can rewrite the intermediate estimate update~\eqref{eqn:gossip_mixing}, 
as follows:
for all $i=1,\ldots,N$,
\begin{align}\label{eqn:gossip_mix}
\hat v_i^k= \sum_{j=1}^{N} [W(k)]_{ij}v^k_j \quad\hbox{for all $k\ge1$,}\quad
\mbox{with} 
\quad v_i^0 = x_i^0,\end{align}
where $x_i^0\in K_i$, $i=1,\ldots,N$, are initial (random) agent decisions.
To rewrite the iterate $x_i^{k+1}$ update (or no update) compactly for all agents, 
we let $\mathbbm{1}_{\{i \in S\}}$ denote the indicator of the event 
$\{i \in S\}$. Then, the update relations in~\eqref{eqn:gossip_up} 
and~\eqref{eqn:gossip_noup} can be written as:
\begin{align}\label{eqn:gossip_x_algo} 
x_i^{k+1} & =\left(\Pi_{K_i}[x_i^k-\alpha_{k,i} F_i(x_i^k,N\hat v^k_i)] - x^k_i 
 \right)\mathbbm{1}_{\{i \in \{I^k,J^k\}\}} 
+ x^k_i, \\
\label{eqn:gossip_estimate_update}
 v_i^{k+1} & =\hat v_i^k+x_i^{k+1}-x_i^k,
\end{align}
Note that only agents $i \in \{I^k,J^k\}$ update since $\mathbbm{1}_{\{i \in \{I^k,J^k\}\}} =0$
when $i \not \in \{I^k,J^k\}$ and, hence,
$x^{k+1}_i = x^k_i$ and $v^{k+1}_i = \hat v^k_i$ with $\hat v^k_i=v^k_i$ (by~\eqref{eqn:gossip_mix}).

We allow agents to use uncoordinated stepsizes that are based on the frequency of the
agent updates. Specifically, agent $i$ uses the stepsize
$\alpha_{k,i} = \frac{1}{\Gamma_k(i)},$
where $\Gamma_k(i)$ denotes the number of updates that agent $i$ has executed 
up to \uvs{ and including at } time $k$. These stepsizes are of the order of $\frac{1}{k}$ 
in a long run~\cite{Ram2009cdc,Nedic2011}. To formalize this result, we need to introduce
the probabilities of agents updates. We let $p_i$ denote the probability of the event 
that agent $i$ updates, i.e. $\{i \in \{I^k,J^{k}\}\},$ for which we have
\[p_i = \frac{1}{N} \left(1 + \sum_{j \in \mathcal{N}_i} p_{ji}\right) \quad 
 \mbox{for all } i \in \mathcal{N},\]
where $p_{ji} >0$ is the probability that agent $i$ is contacted by its neighbor $j$. 
The long term estimates for $\alpha_{k,i}$ 
that we use in our analysis are given in the following lemma
(cf.~\cite{Nedic2011}, Lemma 3), the proof of which is provided in
	the appendix for completeness.

\begin{lemma}\label{lemma:stepsize}
Let Assumption~\ref{asum:graph} hold, and let $\hat p =1+\min_{\{i,j\}\in \Escr}p_{ij}$ and 
$\alpha_{k,i}=1/\Gamma_k(i)$ for all $k$ and $i$. Then, for any $q \in
(0,1/2)$ \uvs{ and for every $\omega \in \Omega$},  
there \uvs{exists} a \uvs{sufficiently large} $\tilde k(\omega) = \tilde k(q,N)$ such that 
we have for all $k \geq \tilde k(\uvs{\omega})$ and $i \in \Nscr$, 
\[
\alpha_{k,i} \leq \frac{2}{kp_i} \mbox{ and } 
\qquad \left|\alpha_{k,i}- \frac{1}{kp_i} \right| \leq \frac{2}{k^{3/2-q}\hat p^2}.
\]
\end{lemma}
\uvs{Note that $\tilde k(\omega)$  is contingent on the sample path
	corresponding to }  $\omega.$ More precisely, we claim the following:
\[
\mathbb{P}\left[\omega: \alpha_{k,i} \leq \frac{2}{kp_i} \mbox{ for } k \geq \tilde k(\omega)\right] = 1.
\]
Based on Lemma~\ref{lemma:stepsize}, we provide \uvs{the next
	corollary without a proof.}
\begin{corollary}\label{cor-alpha} Let Assumption~\ref{asum:graph} hold, and let $\hat p =1+\min_{\{i,j\}\in \Escr}p_{ij}$ and 
$\alpha_{k,i}=1/\Gamma_k(i)$ for all $k$ and $i$. Then for all $i \in
\Nscr$, \uvs{the following hold with probability one:}
\begin{equation*}
(a) \ \sum_{k=1}^\infty \frac{\alpha_{k,i}}{k} < \infty; \quad  (b)
\sum_{k=1}^\infty \alpha_{k,i}^2 < \infty; \quad (c) 
\sum_{k=1}^\infty \left|\alpha_{k,i}- \frac{1}{kp_i} \right| < \infty.
\end{equation*}
\end{corollary}
\uvs{Another useful result is provided in~\cite[Lemma 1]{Ram2009cdc}, and stated below
in a suitable form.}
\begin{lemma}\label{lemma:ram_thm}
Let $\Gscr(\Nscr,\Escr)$ be a graph that satisfies Assumption~\ref{asum:graph}.
Let $W$ be an $N\times N$ random stochastic matrix such that 
$\mathbb{E}[W]$ is doubly stochastic and $\mathbb{E}[W]_{ij} > 0$ 
whenever $\{i, j\}\in\Escr$. Furthermore, let 
the diagonal elements of $W$ be positive almost surely. Then,  for the matrix 
$D= W-\frac{1}{N}\mathbf{1}\mathbf{1}^TW$, there exists
a scalar $\lambda\in(0, 1)$ such that
\[\mathbb{E}[ \|Dz \|^2 ]
\le \lambda \|z\|^2\qquad\hbox{for all }z\in\mathbb{R}^N.\]
\end{lemma}
The random matrices $W(k)$ in~\eqref{eqn:weight_matrix} are in fact doubly stochastic
and thus, $\bar W=\mathbb{E}[W(k)]$ is doubly stochastic. Moreover, it can be easily seen
that $\bar W_{ij} > 0$ whenever $\{i, j\}\in\Escr$.
In addition, $W(k)$ has positive diagonal entries. Hence, Lemma~\ref{lemma:ram_thm}
applies to random matrices $W(k)$. However, since each $W(k)$ is in fact doubly stochastic, 
we have $\mathbf{1}^TW(k)=\mathbf{1}^T$,  implying that
$W(k)-\frac{1}{N}\mathbf{1}\mathbf{1}^TW(k) =W(k)-\frac{1}{N}\mathbf{1}\mathbf{1}^T.$
Hence, using this observation and Lemma~\ref{lemma:ram_thm}, we find that there exists
$\lambda\in(0,1)$ such that for the matrix $D(k)=W(k)-\frac{1}{N}\mathbf{1}\mathbf{1}^T$ we have 
\begin{align}\label{eqn:lam}
\mathbb{E}[ \|D(k)\,z \|^2]\le \lambda \|z\|^2\qquad\hbox{for all }z\in\mathbb{R}^N.
\end{align}
By Jensen's inequality we have 
$|\mathbb{E}[X]|\le \sqrt{\mathbb{E}[X^2]}$ for any random variable $X$ (with a finite expectation),
which when applied to relation~\eqref{eqn:lam} yields
\begin{align}\label{eqn:lamn}
\mathbb{E} [ \| D(k)\,z \| ]\le \sqrt{\lambda}\, \|z\|\qquad\hbox{for all }z\in\mathbb{R}^N.
\end{align}
The value of second largest eigenvalue $\lambda$ controls the rate at which information is dispensed over the network. A network with large $\lambda$ will have players agreeing faster on their estimate of the aggregate than a network with a smaller $\lambda$. In Section~\ref{sec:numerics} we consider a variety of networks to demonstrate the impact of network topology through $\lambda$ on the rate of convergence.

\subsection{Convergence Theory}\label{sec:convergence}
In this section we establish the convergence of 
the asynchronous algorithm~\eqref{eqn:gossip_mix}--\eqref{eqn:gossip_estimate_update}
with the agent specific diminishing stepsize of the form $\alpha_{k,i} = \frac{1}{\Gamma_k(i)}$.
To take account of the history, we introduce $\mathcal{F}_k$ to denote the 
$\sigma-$algebra generated by the entire history up to $k$. More precisely
\[
 \mathcal{F}_k = \mathcal{F}_0 \cup \{I^l,J^l; 1 \leq l \leq k-1 \} \quad \mbox{for all } k \geq 2,
\]
with $\mathcal{F}_{1}=\mathcal{F}_{0}\{x^0_i, i \in \mathcal{N}\}.$
Thus, given $\mathcal{F}_k$, the vectors $v_i^k$ and $x_i^k$ are fully determined.
First we state several result which we will use to claim the convergence of the algorithm,
as well as to analyze the error bounds.

In what follows, we will use a vector-component based analysis. To this end, we introduce
$[z]_\ell$ to denote the $\ell$-th component of a vector 
$z \in \mathbb{R}^n,$ with $\ell=1,\ldots, n.$ 
A component-wise update of each $v^{k+1}_i$ in~\eqref{eqn:gossip_estimate_update} 
is given by: for all $i=1,\ldots,N$,
\[[v^{k+1}_i]_\ell 
= \sum_{j=1}^N[W(k)]_{ij}[v^k_j]_\ell + [x^{k+1}_i -x^k_i]_\ell \quad \mbox{for} \; \ell =1,\ldots,n.
\]
We collect all $\ell$th coordinates of the vectors $v_1^{k},\ldots, v_N^k$ and let
$v^k(\ell)=([v^k_1]_\ell,\ldots,[v^k_N]_\ell)^T$. We similarly do for the 
vectors  $x_1^{k},\ldots, x_N^k$ and let $x^k(\ell)=([x^k_1]_\ell,\ldots,[x^k_N]_\ell)^T$. 
Using the vectors $v^k(\ell)$ and $x^k(\ell)$, we can rewrite the preceding relation as follows:
\begin{align}\label{eqn:lthv}
 v^{k+1}(\ell) = W(k)v^k(\ell) + \zeta^{k+1}(\ell) \quad
 \hbox{with}\quad  \zeta^{k+1}(\ell)=x^{k+1}(\ell) - x^k (\ell)
 \qquad\hbox{for all }\ell=1,\ldots,n.
\end{align}
We have the  following result for $v^{k+1}(\ell)$ for any $\ell$.

\begin{lemma}\label{lem:v_estimate}
Let Assumptions~\ref{assump:set_func}--\ref{assump:Lipschitz} and 
Assumption~\ref{asum:graph} hold. Then, for all $\ell=1,\ldots, n$ and $k\ge0$,
\[\|v^{k+1}(\ell) - [y^{k+1}]_\ell\mathbf{1}\|  
\le  \| D(k)(v^k(\ell) - [y^{k}]_\ell\mathbf{1} ) \|
+ \jk{\sqrt{2} C\sum_{i=1}^N \alpha_{k,i}},\]
where $D(k)=W(k)-\frac{1}{N}\mathbf{1}\mathbf{1}^T$ and 
$C$ is a constant as in Lemma~\ref{cor:fibound}.
\end{lemma}
\begin{proof}
We fix an arbitrary coordinate $\ell$.
By the decision update rule of~\eqref{eqn:gossip_x_algo}, the $i$th coordinate 
of the vector $\zeta^{k+1}(\ell)$ is 
$[\zeta^{k+1}(\ell)]_i 
= [\left(\Pi_{K_i}[x_i^k-\alpha_{k,i} F_i(x_i^k,N\hat v^k_i)] 
- x^k_i\right)\mathbbm{1}_{\{i \in \{I^k,J^k\}\}}]_\ell.$ 
Since $y^{k+1}$ is the average of the vectors 
$v^{k+1}_i$, from~\eqref{eqn:lthv} for the $\ell$th coordinate of this vector we obtain
\[
[y^{k+1}]_\ell = \frac{1}{N}\,\mathbf{1} v^{k+1}(\ell) 
=\frac{1}{N}\left(\mathbf{1}^TW(k) v^k(\ell) + \mathbf{1}^T\zeta^{k+1}(\ell)  \right),
\]
which together with~\eqref{eqn:lthv} leads us to
\begin{align*}
v^{k+1}(\ell) - [y^{k+1}]_\ell\mathbf{1}  
= \left(W(k)-\frac{1}{N}\mathbf{1}\mathbf{1}^TW(k)\right) v^k(\ell)  
+ \left(\mathbb{I}- \frac{1}{N}\mathbf{1}\mathbf{1}^T  \right)\zeta^{k+1}(\ell),
\end{align*}
where $\mathbb{I}$ is the identity matrix. 
Note that each $W(k)$ is a doubly stochastic matrix i.e., 
$W(k)\mathbf{1}=\mathbf{1}$ and $\mathbf{1}^TW(k)=\mathbf{1}^T$. Thus,
$\frac{1}{N}\mathbf{1}\mathbf{1}^TW(k)=\frac{1}{N}\mathbf{1}\mathbf{1}^T$.
Furthermore, we have
$\left(W(k)-\frac{1}{N}\mathbf{1}\mathbf{1}^T\right)\mathbf{1}=0,$ implying that 
\[
\left(W(k)-\frac{1}{N}\mathbf{1}\mathbf{1}^T\right)[y^{k}]_\ell\mathbf{1}=0.
\]
By combining the preceding two relations, using 
$\frac{1}{N}\mathbf{1}\mathbf{1}^TW(k)=\frac{1}{N}\mathbf{1}\mathbf{1}^T$, 
and letting $D(k)=W(k)-\frac{1}{N}\mathbf{1}\mathbf{1}^T$, we obtain
\begin{align*}
v^{k+1}(\ell) - [y^{k+1}]_\ell\mathbf{1}  
= D(k) ( v^k(\ell) - [y^{k}]_\ell\mathbf{1} )
+ \left(\mathbb{I}- \frac{1}{N}\mathbf{1}\mathbf{1}^T  \right)\zeta^{k+1}(\ell).
\end{align*}
Taking \us{norms}, we obtain
\begin{align}\label{eqn:relo}
\|v^{k+1}(\ell) - [y^{k+1}]_\ell\mathbf{1}\|  
\le  \| D(k) ( v^k(\ell) - [y^{k}]_\ell\mathbf{1} ) \|
+ \left\|\left(\mathbb{I}- \frac{1}{N}\mathbf{1}\mathbf{1}^T  \right)\,\zeta^{k+1}(\ell)\right\|.
\end{align}

We next estimate the last term in~\eqref{eqn:relo}.
The matrix $\mathbb{I}- \frac{1}{N}\mathbf{1}\mathbf{1}^T$ is a projection matrix 
(corresponds to the projection on the subspace orthogonal to the vector $\mathbf{1}$),
so $\|\mathbb{I}- \frac{1}{N}\mathbf{1}\mathbf{1}^T \|=1$,
implying that 
\begin{align}\label{eqn:relt}
\left\|\left(\mathbb{I}- \frac{1}{N}\mathbf{1}\mathbf{1}^T \right)\zeta^{k+1}(\ell)\right\|
\le \left\|\mathbb{I}- \frac{1}{N}\mathbf{1}\mathbf{1}^T \right\|\,\|\zeta^{k+1}(\ell)\|
=\|\zeta^{k+1}(\ell)\|.
\end{align}
From the definition of $\zeta^{k+1}(\ell)$ in~\eqref{eqn:lthv}, we see that
\[\|\zeta^{k+1}(\ell)\|^2
 =\sum_{i \in \{I^k,J^k\}} 
 \left\|\left[ \Pi_{K_i}[x_i^k-\alpha_{k,i} F_i(x_i^k,N\hat v^k_i)] - x^k_i \right]_\ell\right\|^2.\]
Using the non-expansive property of the projection operator and 
\jk{$$\alpha_{k,i}^2 \leq \sum_{i \in \{I^k,J^k\}} \alpha_{k,i}^2 \; \leq \sum_{i=1}^N\alpha_{k,i}^2,$$} we have 
\[\|\zeta^{k+1}(\ell)\|^2 
\leq  \sum_{i \in \{I^k,J^k\}} \|\alpha_{k,i}[F_i(x_i^k,N\hat v^k_i)]_\ell\|^2 \leq 
\sum_{i=1}^N \alpha_{k,i}^2\; \sum_{i \in \{I^k,J^k\}} \|F_i(x_i^k,N\hat v^k_i)\|^2.
\]
By \us{Lemma}~\ref{cor:fibound}, $\|F_i(x_i^k,N\hat v^k_i)\| \leq C$ for all $i,k$ and some $C>0$. 
This and $|\{I^k,J^k\}| =2$ imply
$\|\zeta^{k+1}(\ell)\|^2 \leq 2 C^2 \sum_{i=1}^N \alpha_{k,i}^2.$
By taking square roots we obtain
\jk{$$\|\zeta^{k+1}(\ell)\|\leq \sqrt{2} C \sqrt{\sum_{i=1}^N \alpha_{k,i}^2}\leq \sqrt{2} C\sum_{i=1}^N \alpha_{k,i}$$}
which when combined with~\eqref{eqn:relo} and~\eqref{eqn:relt} yields
\[\|v^{k+1}(\ell) - [y^{k+1}]_\ell\mathbf{1}\|  
\le  \| D(k) ( v^k(\ell) - [y^{k}]_\ell\mathbf{1} ) \|
+\jk{\sqrt{2}\, C \sum_{i=1}^N \alpha_{k,i}}.\]
 \end{proof}

Our result involves the average $y^k$ of the estimates $v^k_i,i\in\Nscr$, which will be important in
establishing the convergence of the algorithm.

\begin{lemma}\label{lem:conv_estimate}
Let Assumptions~\ref{assump:set_func}--\ref{assump:Lipschitz} and Assumption~\ref{asum:graph} hold. 
Let $v^k_i$ be given by~\eqref{eqn:gossip_mix} and~\eqref{eqn:gossip_estimate_update}, respectively, 
and let  $y^k=\frac{1}{N}\sum_{i=1}^N v^k_i$.
Then, we have
 \[
   \sum_{k=1}^\infty 
   \frac{1}{k} \,\|v^k_i-y^{k}\|^2   < \infty
   \quad \mbox{a.s.\ for all $i\in\Nscr$.}
 \]
\end{lemma}

\begin{proof}
From Lemma~\ref{lem:v_estimate},
we obtain \uvs{that the following holds for $k \geq 0$} almost surely,
\begin{align*}
\|v^{k+1}(\ell) - [y^{k+1}]_\ell\mathbf{1}\|  
\le  \| D(k) ( v^k(\ell) - [y^{k}]_\ell\mathbf{1} ) \|
+ \jk{\sqrt{2}\, C \sum_{i=1}^N \alpha_{k,i}}.
\end{align*}
By taking conditional expectations with respect to $\mathcal{F}_k$,
   we obtain \uvs{that the following holds almost surely for all $k$:} 
\begin{align}\label{eqn:relf}
\mathbb{E}[\|v^{k+1}(\ell) - [y^{k+1}]_\ell\mathbf{1}\|\mid \mathcal{F}_k ] 
 \leq  \mathbb{E}[ \| D(k)(v^k(\ell) - [y^{k}]_\ell\mathbf{1} ) \|\,\left|\right.\, \mathcal{F}_k ]
 + \uvs{\sqrt{2}\, C \mathbb{E}\left[ \sum_{i=1}^N \alpha_{k,i}\mid
	 \mathcal{F}_k\right]}.
\end{align}
Note that the expectation in the term on the right hand side
is taken with respect to the randomness in the matrix $W(k)$ only.
By relation~\eqref{eqn:lamn}, we have
\[\mathbb{E}[\| D(k) ( v^k(\ell) - [y^{k}]_\ell\mathbf{1} ) \|\mid \mathcal{F}_k]
\le \sqrt{\lambda}\| v^k(\ell) - [y^{k}]_\ell\mathbf{1}\|,\]
 which combined with~\eqref{eqn:relf} yields \uvs{that the following
	 holds for all $k$ in an  almost sure sense:}  
\begin{align}\label{eqn:rels}
\mathbb{E}[ \|v^{k+1}(\ell) - [y^{k+1}]_\ell\mathbf{1}\| \mid \mathcal{F}_k ]\leq 
\sqrt{\lambda}\| v^k(\ell)-[y^{k}]_\ell\mathbf{1}\| + \jk{\sqrt{2}\, C\, \mathbb{E}\left[\sum_{i=1}^N \alpha_{k,i} \mid \mathcal{F}_k \right]}.
\end{align}

By \us{multiplying} both sides of~\eqref{eqn:rels} with $\frac{1}{k}$
and by using $\frac{1}{k+1}<\frac{1}{k}$ we find that 
almost surely for all $k$: 
\begin{align*}
 \frac{1}{k+1}\mathbb{E}[ \|v^{k+1}(\ell) - [y^{k+1}]_\ell\mathbf{1}\| \mid \mathcal{F}_k ]
 &\leq \frac{\sqrt{\lambda}}{k}\| v^k(\ell)-[y^{k}]_\ell\mathbf{1}\| 
 + \frac{\jk{\sqrt{2}\, C\, \mathbb{E}[\sum_{i=1}^N \alpha_{k,i} \mid \mathcal{F}_k ]}}{k}\cr
 &= \frac{1}{k}\| v^k(\ell)-[y^{k}]_\ell\mathbf{1}\| 
 - \frac{1-\sqrt{\lambda}}{k}\| v^k (\ell) - [y^{k}]_\ell\mathbf{1}\| \\
& + \jk{\sqrt{2}\,C\,\mathbb{E}\left[\sum_{i=1}^N\frac{\alpha_{k,i}}{k}\mid \mathcal{F}_k \right]}.
\end{align*}
Since $\lambda \in (0,1)$ we have $1-\sqrt{\lambda}>0$ \uvs{and 
	  $\sum_{k} \mathbb{E}\left[\sum_{i=1}^N\frac{\alpha_{k,i}}{k}\mid
	  \mathcal{F}_k \right] < \infty$ since $\sum_{k} \sum_{i=1}^N \frac{2}{k^2 p_i} <
	  \infty$ with probability one from Corollary~\ref{cor-alpha},
	  allowing for invoking the} supermartingale convergence lemma
(Lemma~\ref{lemma:supermartingale}). 
Therefore, we may conclude that
\[
 \sum_{k=1}^\infty  \frac{1}{k} \|v^k(\ell)-[y^{k}]_\ell\mathbf{1} \| < \infty \quad \mbox{a.s.}
\]
Recalling that $v^k(\ell)=([v_1^k]_\ell,\ldots,[v_N^k]_\ell)^T$, 
the preceding relation implies that 
\[
 \sum_{k=1}^\infty  \frac{1}{k} \left|[v_i^k]_\ell-[y^{k}]_\ell\right|< \infty \quad 
 \mbox{for all $i\in\Nscr$ a.s.}
\]
The coordinate index $\ell$ was arbitrary, so the relation is
also true for every coordinate index $\ell=1,\ldots, n$. In particular,
 since $\|v_i^k - y^{k}\|\le \sum_{\ell=1}^n|[v_i^k]_\ell - [y^{k}]_\ell|$
 we have
 \[\sum_{k=1}^\infty  \frac{1}{k} \left\|v_i^k - y^{k} \right\|
 \le \sum_{k=1}^\infty \frac{1}{k} \sum_{\ell=1}^n|[v_i^k]_\ell - [y^{k}]_\ell| < \infty
 \quad \mbox{for all $i\in\Nscr$ a.s.}
\]
\end{proof}

For the rest of the paper, we use $x^k$ to denote the vector with components $x_i^k$, $i=1,\ldots,N$,
i.e., $x^k=(x_1^k,\ldots,x_N^k)$ and we write $x^*$ for the vector $(x_1^*,\ldots,x^*_N)$.
We now show the convergence of the algorithm. 
We have the following result,
where $x^*$ denotes the unique Nash equilibrium
of the aggregative game in~\eqref{eqn:problem}. 

\begin{proposition}\label{prop:gs_convergence}
Let Assumptions~\ref{assump:set_func}--\ref{assump:Lipschitz} and Assumption~\ref{asum:graph}
hold. Then, the sequence $\{x^k\}$ generated by 
the method~\eqref{eqn:gossip_mix}--\eqref{eqn:gossip_estimate_update}
 with the stepsize $\alpha_{k,i}=\frac{1}{\Gamma_k(i)}$ converges to the (unique) $x^*$ of 
 the game almost surely.
\end{proposition}
\begin{proof}
Under strict monotonicity of the mapping and the compactness of
	$K$, uniqueness of the equilibrium follows from
		Proposition~\ref{uniq-equil}.
Then, by the definition of $x^{k+1}_i$ we have
\begin{align*}
  \|x^{k+1}_i - x^*_i\|^2 
  = \left\|\left(\Pi_{K_i}[x_i^k-\alpha_{k,i} F_i(x_i^k,N\hat v^k_i)] - x^k_i\right)
  \mathbbm{1}_{\{i \in \{I^k,J^{k}\}\}}
  + x^k_i - x^*_i\right\|^2.
\end{align*}
Using $x^*_i = \Pi_{K_i}[x_i^*-\alpha_{k,i} F_i(x^*_i,\bar x^*)]$ and the non-expansive property of the projection operator,
we have for $i \in \{I^{k},J^{k}\},$
\begin{align*}
  \|x^{k+1}_i - x^*_i\|^2 
 & \leq \|x_i^k-\alpha_{k,i} F_i(x_i^k,N\hat v^k_i) - x^k_i - x^*_i + \alpha_{k,i} F_i(x^*_i,\bar x^*)\|^2 \nonumber \\
 & =  \|x_i^k-x^*_i \|^2 + \alpha_{k,i}^2\| F_i(x_i^k,N\hat v^k_i) -F_i(x^*_i,\bar x^*) \|^2   \cr
 &-2\alpha_{k,i} (F_i(x_i^k,N\hat v^k_i) -F_i(x^*_i,\bar x^*))^T(x^k_i-x^*_i).
\end{align*}
\uvs{By expressing $\alpha_{k,i}$ as} $\alpha_{k,i}=\left(\alpha_{k,i}- \frac{1}{kp_i}\right)
+\frac{1}{kp_i}$, we have the following 
for all \jk{$k$}:
\begin{align}
\label{eqn:gossip_iter_rel1}
  \|x^{k+1}_i - x^*_i\|^2 
 & = \|x_i^k-x^*_i \|^2 + \alpha_{k,i}^2\| F_i(x_i^k,N\hat v^k_i) -F_i(x^*_i,\bar x^*) \|^2 \nonumber\\
 & -2 \left(\alpha_{k,i}- \frac{1}{kp_i} \right)
 (F_i(x_i^k,N\hat v^k_i) -F_i(x^*_i,\bar x^*))^T(x^k_i-x^*_i)\nonumber \\
 & -\frac{2}{kp_i} (F_i(x_i^k,N\hat v^k_i) -F_i(x^*_i,\bar x^*))^T(x^k_i-x^*_i)\cr
 & \leq \|x_i^k-x^*_i \|^2 
 + \jk{\alpha_{k,i}^2}\| F_i(x_i^k,N\hat v^k_i) -F_i(x^*_i,\bar x^*) \|^2 \nonumber \\
 & +\jk{2\left|\alpha_{k,i}- \frac{1}{kp_i} \right|} \left| (F_i(x_i^k,N\hat v^k_i) -F_i(x^*_i,\bar x^*))^T(x^k_i-x^*_i)\right| \nonumber \\
 & -\frac{2}{kp_i} (F_i(x_i^k,N\hat v^k_i) -F_i(x^*_i,\bar x^*))^T(x^k_i-x^*_i).
\end{align}
By \us{Lemma}~\ref{cor:fibound} and Assumption~\ref{assump:set_func} we can see
that $\| F_i(x_i^k,N\hat v^k_i) -F_i(x^*_i,\bar x^*) \|^2 \le C_1$ for some scalar $C_1$,
and for all $i$ and $k$.  
Similarly, for the term in~\eqref{eqn:gossip_iter_rel1} involving the absolute value,
we can see that $|(F_i(x_i^k,N\hat v^k_i) -F_i(x^*_i,\bar x^*))^T(x_i^k-x_i^*)|\le C_2$
for some scalar $C_2$, and for all $i$ and $k$. 
Substituting these estimates in~\eqref{eqn:gossip_iter_rel1},  
we obtain
\begin{align}\label{eqn:rr1}
\|x^{k+1}_i - x^*_i\|^2 
 & \leq \|x_i^k-x^*_i \|^2 
 +\jk{\alpha_{k,i}^2C_1 + 2C_2\left|\alpha_{k,i}- \frac{1}{kp_i}
	 \right|} \notag \\
 & -\frac{2}{kp_i}(F_i(x_i^k,N\hat v^k_i) -F_i(x^*_i,\bar x^*))^T(x^k_i-x^*_i).
\end{align}
For the last term in the preceding relation, by adding and subtracting 
$F_i(x^k_i,Ny^k)$ and using $Ny^k = \sum_{i=1}^N x^k_i=\bar x^k$ 
(cf. Lemma~\ref{lemma:true_agg}), we write
\begin{align*}
(F_i(x_i^k,N\hat v^k_i) -F_i(x^*_i,\bar x^*))^T(x_i^k-x_i^*)
 & =  (F_i(x_i^k,N\hat v^k_i) - F_i(x^k_i,Ny^k))^T(x_i^k-x_i^*)\cr
 & + (F_i(x^k_i,\bar x^k) - F_i(x^*_i,\bar x^*) )^T(x_i^k-x_i^*)\cr
 & \ge -\|F_i(x_i^k,N\hat v^k_i) - F_i(x^k_i,Ny^k)\|\,\|x_i^k-x_i^*\|\cr
 &+ (F_i(x^k_i,\bar x^k) - F_i(x^*_i,\bar x^*) )^T(x_i^k-x_i^*)\cr
 & \ge  -\bar L_iN\|\hat v^k_i - y^k\|M
 + (F_i(x^k_i,\bar x^k) - F_i(x^*_i,\bar x^*) )^T(x_i^k-x_i^*)
\end{align*}
where we use the Lipschitz property of the mapping $F_i$
(Assumption~\ref{assump:Lipschitz}), while $M$ is a constant such that 
$\max_{x_i,z_i\in K_i}\|x_i-z_i\|\le M$ for all $i$.
The vector $\hat v_i^k$ is a convex combination of $v_j^k$ \uvs{over $j = 1,
	\hdots, N$}(cf.~\eqref{eqn:gossip_mix}). Therefore,  
by the convexity of the norm, we have 
$\|\hat v^k_i - y^k\|\le \sum_{j=1}^N[W(k)]_{ij}\|v_j^k-y^k\|$, which yields
\begin{align}\label{eqn:rr2}
(F_i(x_i^k,N\hat v^k_i) -F_i(x^*_i,\bar x^*))^T(x_i^k-x_i^*)
 & \ge  -\bar L_iNM \sum_{j=1}^N[W(k)]_{ij}\| v^k_j - y^k\|\cr
 & + (F_i(x^k_i,\bar x^k) - F_i(x^*_i,\bar x^*) )^T(x_i^k-x_i^*).\end{align}
Finally, by combining relations~\eqref{eqn:rr1} and~\eqref{eqn:rr2} we obtain
for  $i\in\{I^k,J^{k}\}$ and for all $k \geq 0$,
\begin{align*}
\|x^{k+1}_i - x^*_i\|^2 
 & \leq \|x_i^k-x^*_i \|^2 
 +\jk{\alpha_{k,i}^2C_1 + 2C_2\left|\alpha_{k,i}- \frac{1}{kp_i} \right|}
 + \frac{2}{kp_i}\bar L_iNM \sum_{j=1}^N[W(k)]_{ij}\|v^k_j - y^k\|\cr
&  -\frac{2}{kp_i}(F_i(x_i^k,\bar x^k) -F_i(x^*_i,\bar x^*))^T(x^k_i-x^*_i).
\end{align*}
Since $x^{k+1}_i=x^k_i$ when $i \not \in \{I^k,J^{k}\}$, it follows that 
$\|x^{k+1}_i - x^*_i\|^2 =  \|x_i^k-x^*_i \|^2$ for $i \not \in \{I^k,J^{k}\}$.
We combine these two cases with the fact that agent $i$ updates with probability $p_i$ 
and, thus obtain almost surely  for all $i \in \mathcal{N}$ and \jk{for all $k$}  
\begin{align}\label{eqn:est3}
 \mathbb{E}[\|x^{k+1}_i - x^*_i\|^2 \mid \mathcal{F}_{k}]
& \le  \|x_i^k-x^*_i \|^2 
 +\jk{C_1\mathbb{E}[\alpha_{k,i}^2 \mid \mathcal{F}_{k}] +
	 2C_2\mathbb{E}\left[\left|\alpha_{k,i}- \frac{1}{kp_i} \right| \mid
		 \mathcal{F}_{k}\right]} \\
& \notag
 + \frac{2}{k}\bar L_iNM\sum_{j=1}^N[W(k)]_{ij}\|v^k_j - y^k\|  -\frac{2}{k}(F_i(x_i^k,\bar x^k) -F_i(x^*_i,\bar x^*))^T(x^k_i-x^*_i).
 \end{align}
Summing relations~\eqref{eqn:est3} 
over $i=1,\ldots,N$, using the fact that $W(k)$ is doubly stochastic
 and recalling that $F_i$ are coordinate maps for $F$ 
and $F(x,\bar x)$ defines $\phi$ (cf.~\eqref{eqn:svinot} and~\eqref{eqn:phimap}), 
we further obtain for all $k$: 
\begin{align}\label{eqn:kraj}
 \mathbb{E}[\|x^{k+1} - x^*\|^2 \mid \mathcal{F}_{k}]
&  \le  \|x^k-x^*\|^2 +\jk{C_1\mathbb{E}\left[\sum_{i=1}^N\alpha_{k,i}^2\mid \mathcal{F}_{k}\right] + 2C_2\mathbb{E}\left[\sum_{i=1}^N\left|\alpha_{k,i}- \frac{1}{kp_i} \right|\mid \mathcal{F}_{k}\right]}
\cr
& + \frac{2}{k}\bar L_iNM\sum_{j=1}^N \|v^k_j - y^k\| -\frac{2}{k}(\phi(x^k) - \phi(x^*))^T(x^k-x^*),
 \end{align}
where $p_{\min}=\min_{i}p_i$ and $p_{\max}=\max_i p_i$.
We now verify that we can apply the supermartingale convergence result 
(cf.~Lemma~\ref{lemma:supermartingale}) to 
relation~\eqref{eqn:kraj}. \jk{From Corollary~\ref{cor-alpha} it follows that 
\[
\sum_{k=1}^\infty \left(C_1\mathbb{E}\left[\sum_{i=1}^N\alpha_{k,i}^2\mid \mathcal{F}_{k}\right] + 2C_2\mathbb{E}\left[\sum_{i=1}^N\left|\alpha_{k,i}- \frac{1}{kp_i} \right|\mid \mathcal{F}_{k}\right]\right) < \infty.
\]}
Further from Lemma~\ref{lem:conv_estimate}  it follows that 
$\sum_{k=1}^\infty 
\sum_{i=1}^N\frac{1}{k} \|v^k_i-y^k\| < \infty$ almost surely.
Thus,  all conditions of Lemma~\ref{lemma:supermartingale} are satisfied 
and we conclude that 
\begin{equation}\label{eqn:seq_conv}
\{\|x^k-x^*\|\} \quad \mbox{converges {\it a.s.},}
\end{equation}
\begin{equation}\label{eqn:F_conv}
 \sum_{k=1}^\infty \frac{2}{k} (\phi(x) - \phi(x^*))^T(x^k-x^*) < \infty
\quad \mbox{{\it a.s.}}
\end{equation}
Since $\{x^k\}\subset K$ and $K$ is compact (Assumption~\ref{assump:set_func}),
it follows that $\{x^k\}$ has an accumulation point in $K$.
By $\sum_{k=1}^\infty\frac{1}{k} = \infty$, relation~\eqref{eqn:F_conv} implies that
$(x^k-x^*)^T(\phi(x^{k}) - \phi(x^*))\to0$ along a subsequence almost surely, 
say $\{x^{k_\ell}\}$. Then, by the strict monotonicity of $\phi$, it follows that 
$\{x^{k_\ell}\}\to x^*$ as $\ell\to\infty$  almost surely. 
By~\eqref{eqn:seq_conv},
the entire sequence converges to $x^*$ almost surely.
\end{proof}

\subsection{Error Bounds for Constant Stepsize}\label{sec:4.3}
In this section, we investigate the properties 
of the algorithm when agents employ a deterministic
constant, albeit uncoordinated, stepsize. More specifically, our 
interest lies in establishing error bounds contingent on the
deviation of stepsize across agents. Under this setting, the stepsize is $\alpha_{k,i}=\alpha_i$
in the update rule for agents' decisions in~\eqref{eqn:gossip_x_algo}, which reduces to 
\begin{align*}
x_i^{k+1}=\left(\Pi_{K_i}[x_i^k-\alpha_{i} F_i(x_i^k,N\hat v^k_i)] - x^k_i 
 \right)\mathbbm{1}_{\{i \in \{I^k,J^k\}\}} 
+ x^k_i,
\end{align*}
where $\alpha_{i}$ is a positive constant stepsize for agent $i$. It is 
worth mentioning that the \us{rules for} mixing estimates ~\eqref{eqn:gossip_mixing} 
and \us{updating estimates}~\eqref{eqn:gossip_estimate_update} are invariant 
under this modification. Also, we allow agents to independently choose 
$\alpha_{i},$ thereby maintaining the complete decentralization feature of the gossip algorithm.  
We begin by providing an updated estimate for the disagreement among the agents. 
Our result \us{is analogous to that} of Lemma~\ref{lem:conv_estimate}.

\begin{lemma}\label{lem:conv_estimate_const}
Let Assumptions~\ref{assump:set_func}--\ref{assump:Lipschitz} and \ref{asum:graph} hold.  
Consider $\{v^k_i\}$, $i =1,\ldots,N,$ that are generated by algorithm 
in~\eqref{eqn:gossip_mix}--\eqref{eqn:gossip_estimate_update} with $\alpha_{k,i}=\alpha_i$.
 Then, for $y^k=\frac{1}{N}\sum_{i=1}^N v_i^k$ we have 
  \[  \limsup_{k \to \infty }
  \sum_{i=1}^N \mathbb{E}[\|v^{k}_i-y^{k}\|^2]  \leq 
   \frac{2n\alpha_{\max}^2 C^2}{(1-\sqrt{\lambda})^2}, 
    \qquad
   \limsup_{k \to \infty }\sum_{i=1}^N \mathbb{E}[\|v^{k}_i-y^{k}\|]  \leq 
   \frac{\sqrt{2nN}\alpha_{\max} C}{1-\sqrt{\lambda} },
   \qquad  \mbox{a.s.,}
 \]
 where $\alpha_{\max} = \max_{i} \{\alpha_i\}$, $C$ is the constant as
 in \us{Lemma}~\ref{cor:fibound},
  and $\lambda$ is as given in~\eqref{eqn:lam}.
\end{lemma}
\begin{proof} We fix an arbitrary index $\ell$.
By~Lemma~\ref{lem:v_estimate} with $\alpha_{k,i}=\alpha_i$, we have for all $k\ge0$,
\begin{align}\label{eqn:base1}
\|v^{k+1}(\ell) - [y^{k+1}]_\ell\mathbf{1}\|  
\le  \|D(k) ( v^k(\ell) - [y^{k}]_\ell\mathbf{1} )\|
+ \sqrt{2} C\alpha_{\max},\end{align}
where $D(k)=W(k)-\frac{1}{N}\mathbf{1}\mathbf{1}^T$ and $\alpha_{\max}= \max_i\alpha_{i}$.
Note that by relation~\eqref{eqn:lamn} we have
\begin{align}\label{eqn:exp1}
\mathbb{E}[\|D(k)(v^k(\ell) - [y^{k}]_\ell\mathbf{1}) \|]
=\mathbb{E}\left[\mathbb{E}[\|D(k)(v^k(\ell) - [y^{k}]_\ell\mathbf{1}) \|\mid\mathcal{F}_k]\right]
\le \sqrt{\lambda}\,\mathbb{E}[\|v^k(\ell) - [y^{k}]_\ell\mathbf{1} \|].\end{align}
Thus, by taking the expectation of \us{both} sides in~\eqref{eqn:base1}, we obtain
\begin{align*}
\mathbb{E}\left[\|v^{k+1}(\ell) - [y^{k+1}]_\ell\mathbf{1}\|\right]
&\le \sqrt{\lambda}\mathbb{E}\left[\|v^k(\ell) - [y^{k}]_\ell\mathbf{1}\|\right]
+ \sqrt{2} C\alpha_{\max}\qquad\hbox{for all }k\ge0,
\end{align*}
which by iterative recursion leads to
\begin{align*}
\mathbb{E}\left[\|v^{k+1}(\ell) - [y^{k+1}]_\ell\mathbf{1}\|\right]
&\le \left(\sqrt{\lambda}\right)^{k+1}\mathbb{E}\left[\|v^0(\ell) - [y^{0}]_\ell\mathbf{1}\|\right]
+ \sqrt{2} C\alpha_{\max}\sum_{s=0}^k(\sqrt{\lambda})^s\qquad\hbox{for all }k\ge0.
\end{align*}
Thus, by letting $k\to\infty$, we obtain the following limiting result
\begin{equation}
\label{eqn:misalign_const}
\limsup_{k \to \infty} \mathbb{E}[ \|v^{k+1}(\ell)- [y^{k+1}]_\ell\mathbf{1}\| ]
\leq  \frac{\sqrt{2}C\alpha_{\max}}{1-\sqrt{\lambda}}.
\end{equation}
By taking the squares of both sides in relation~\eqref{eqn:base1}, we find
\begin{align*}
\|v^{k+1}(\ell) - [y^{k+1}]_\ell\mathbf{1}\|^2  
\le  \|D(k)(v^k(\ell) - [y^{k}]_\ell\mathbf{1}) \|^2 +
2\sqrt{2} C\alpha_{\max} \|D(k)(v^k(\ell) - [y^{k}]_\ell\mathbf{1}) \|
+2 C^2\alpha^2_{\max}.\end{align*}
Taking the expectation on both sides in the preceding relation and using
estimate~\eqref{eqn:exp1},
we obtain
\begin{align}\label{eqn:basesq}
\mathbb{E}[\|v^{k+1}(\ell) - [y^{k+1}]_\ell\mathbf{1}\|^2]  
&\le \mathbb{E}[\|D(k)(v^k(\ell) - [y^{k}]_\ell\mathbf{1})\|^2] +
2\sqrt{2} C\alpha_{\max} \sqrt{\lambda}\mathbb{E}[\|v^k(\ell) - [y^{k}]_\ell\mathbf{1}\|]
+2 C^2\alpha^2_{\max}\notag \\
&\le \lambda\mathbb{E}[\|v^k(\ell) - [y^{k}]_\ell\mathbf{1}\|^2] +
2\sqrt{2} C\alpha_{\max} \sqrt{\lambda}\mathbb{E}[\|v^k(\ell) \notag \\
& - [y^{k}]_\ell\mathbf{1}\|]
+2 C^2\alpha^2_{\max}
\end{align}
where the last inequality follows by 
\[\mathbb{E}[\|D(k)(v^k(\ell) - [y^{k}]_\ell\mathbf{1}) \|^2]
=\mathbb{E}\left[\mathbb{E}[\|D(k)(v^k(\ell) - [y^{k}]_\ell\mathbf{1}) \|^2\mid\mathcal{F}_k]\right]
\le \lambda\,\mathbb{E}[\|v^k(\ell) - [y^{k}]_\ell\mathbf{1} \|^2],\]
which is a consequence of relation~\eqref{eqn:lam}. 
Since $v_i^k$ and $y^{k}$ are convex combinations of points in 
$K_1,\ldots,K_N$ and each $K_i$ is compact, the sequence $\{\|v^{k+1}(\ell) - [y^{k+1}]_\ell\mathbf{1}\|\}$ is bounded, implying that so is the sequence $\{\mathbb{E}[\|v^{k+1}(\ell) - [y^{k+1}]_\ell\mathbf{1}\|]\}$.
Thus, $\limsup_{k \to \infty} \mathbb{E}[ \|v^{k+1}(\ell)- [y^{k+1}]_\ell\mathbf{1}\|^2 ]$ exists,
and let us denote this limit by $S$. Letting $k\to\infty$ in relation~\eqref{eqn:basesq} and 
using~\eqref{eqn:misalign_const}, we obtain
\[S\le 
\lambda S +
\left(\frac{2\sqrt{\lambda}}{1-\sqrt{\lambda}}+1\right)2 C^2\alpha^2_{\max}
=\lambda S +\frac{1+\sqrt{\lambda} }{1-\sqrt{\lambda} }
2 C^2\alpha^2_{\max},\]
which upon solving for $S$ and recalling the notation yields
\[\limsup_{k \to \infty} \mathbb{E}[ \|v^{k+1}(\ell)- [y^{k+1}]_\ell\mathbf{1}\|^2 ]
\le \frac{1+\sqrt{\lambda} }{(1-\lambda)(1-\sqrt{\lambda}) }2 C^2\alpha^2_{\max}
= \frac{1}{(1-\sqrt{\lambda})^2 }2 C^2\alpha^2_{\max}.\]
The preceding relation is true for any $\ell$. Thus, since limsup is invariant 
under the index-shift, we have
\[\limsup_{k \to \infty} \sum_{\ell=1}^n\mathbb{E}[ \|v^{k}(\ell)- [y^{k}]_\ell\mathbf{1}\|^2 ]
\le \sum_{\ell=1}^n\limsup_{k \to \infty} \mathbb{E}[\|v^{k}(\ell)- [y^{k}]_\ell\mathbf{1}\|^2 ]
\le \frac{2nC^2\alpha^2_{\max}}{(1-\sqrt{\lambda})^2 },\]
and  by the linearity of the expectation, it follows
\[\limsup_{k \to \infty} \mathbb{E}\left[ \sum_{\ell=1}^n\|v^{k}(\ell)- [y^{k}]_\ell\mathbf{1}\|^2 \right]
\le \frac{2nC^2\alpha^2_{\max}}{(1-\sqrt{\lambda})^2 }.\]
Recalling that vector $v^{k}(\ell)$ consists of the $\ell$th coordinates of the vectors
$v_1^k,\ldots,v_N^k$ (i.e., $v^k(\ell)=([v_1^k]_\ell,\ldots,[v_N^k]_\ell)^T$), 
we see that $\|\sum_{\ell=1}^n\|v^{k}(\ell)- [y^{k}]_\ell\mathbf{1}\|^2=
\sum_{i=1}^N\|v^k - y^k\|^2$. Hence, the preceding relation is equivalent to
\[\limsup_{k \to \infty} \sum_{i=1}^N\mathbb{E}[ \|v_i^{k} - y^{k}\|^2]
\le \frac{2nC^2\alpha^2_{\max}}{(1-\sqrt{\lambda})^2 },\]
which is the first relation stated in the lemma.
In particular, the preceding relation implies that
\begin{align}\label{eqn:estim1}
\limsup_{k \to \infty} \sqrt{\sum_{i=1}^N\mathbb{E}[ \|v_i^{k} - y^{k}\|^2]}
\le \frac{\sqrt{2n}\,C\alpha_{\max} }{1-\sqrt{\lambda}}.\end{align}
On the other hand, by Holders' inequality we have 
\begin{align*}
\sum_{i=1}^N\mathbb{E}[ \|v_i^{k} - y^{k}\|]
\le\sqrt{N}\sqrt{\sum_{i=1}^N \mathbb{E}[ \|v_i^{k} - y^{k}\|^2]}.\end{align*}
from which by taking the limit as $k\to\infty$ and using~\eqref{eqn:estim1},
we obtain 
\[\limsup_{k \to \infty }
  \sum_{i=1}^N \mathbb{E}[\|v^k_i-y^{k}\|]  
  \leq \frac{\sqrt{2nN}\,C\alpha_{\max} }{1-\sqrt{\lambda}}.\]
\end{proof}

We now estimate the limiting error of the algorithm under the
	additional assumption of strong monotonicity of the mapping
		$\phi$. For this result, we \us{also} assume an additional Lipschitz property
	for the maps $F_i$, as given below. 
\begin{assumption}\label{assump:Lipschitz_more}
Each mapping $F_i(x_i, u)$ is uniformly Lipschitz continuous in 
$x_i$ over $K_i$, for every fixed
   $u\in \bar K$ i.e., for some $L_i>0$ and for all $x_i,y_i\in K_i$,
   \[
   \|F_i(x_{i},u)-F_i(y_{i}, u)\| \leq L_i\|x_{i}-y_{i}\|.
   \]
\end{assumption}

We have the following result. 
 
\begin{proposition}\label{prop:gs_convergence_const}
Let Assumptions~\ref{assump:set_func}--\ref{assump:Lipschitz}, 
\ref{asum:graph}, and~\ref{assump:Lipschitz_more} hold, and let 
the mapping $\phi$ be strongly monotone over the set $K$ with a constant $\mu>0$,
in the following sense:
\[(\phi(x)-\phi(y))^T(x-y) \geq \mu \|x-y\|^2\qquad\hbox{for all $x, y\in K$}.\]
Consider the sequence $\{x^k\}$ 
generated by the method~\eqref{eqn:gossip_mix}--\eqref{eqn:gossip_estimate_update} with 
$\alpha_{k,i}=\alpha_i$. Suppose that the stepsizes $\alpha_i$ are such that
\begin{align}\label{eqn:step_cond_rel}
0<1-2\mu p_{\min}\alpha_{\min}+2p_{\max} (\max_i L_i)(\alpha_{\max}-\alpha_{\min})<1,
\end{align}
where for $i=1,\ldots,N$, $L_i$ are Lipshitz constants from Assumption~\ref{assump:Lipschitz_more}, 
$\displaystyle \alpha_{\max} = \max_{i}\alpha_i,$ 
$\displaystyle \alpha_{\min} = \min_{i} \alpha_i,$
$\displaystyle p_{\max} = \max_{i} p_i$, and
$\displaystyle p_{\min} = \min_{i} p_i$.
Then, the following result holds
\[\limsup_{k\to\infty}\mathbb{E}[ \|x^{k} - x^*\|^2]
 \le 
 \frac{p_{\max}\alpha_{\max}^2 \left( 2C^2N
 + BC \frac{\sqrt{2nN}}{1-\sqrt{\lambda}} \right) }
 {\mu p_{\min}\alpha_{\min} - p_{\max} (\max_i L_i)(\alpha_{\max}-\alpha_{\min})},\]
 where $x^*$ is the unique solution of VI$(K,\phi)$, 
 $C$ is as in \us{Lemma}~\ref{cor:fibound}, $\lambda$ is as in~\eqref{eqn:lam}, and
 $B=(\max_i \bar L_i)NM$ with $\bar L_i$, $i=1,\ldots,N$, being the Lipschitz constants from 
 Assumption~\ref{assump:Lipschitz}, and $M\ge \max_{x_i,z_i\in K_i}\|x_i-z_i\|$ for all $i$.

\end{proposition}
\begin{proof}
Since the map is strongly monotone, there is a unique solution $x^*\in K$ to VI$(K,\phi)$ 
(see Theorem~2.3.3.\ in~\cite{Facchinei_Pang03}).
Then, by the definition of $x^{k+1}_i$ we have
\begin{align*}
  \|x^{k+1}_i - x^*_i\|^2 
  = \left\|\left(\Pi_{K_i}[x_i^k-\alpha_{i} F_i(x_i^k,N\hat v^k_i)] 
  - x^k_i\right)\mathbbm{1}_{\{i \in \{I^k,J^{k}\}\}}
  + x^k_i - x^*_i\right\|^2.
\end{align*}
Using $x^*_i = \Pi_{K_i}[x_i^*-\alpha_{i} F_i(x^*_i,\bar x^*)]$ and the non-expansive property 
of the projection operator, we have for $i \in \{I^{k},J^{k}\},$
\begin{align}
\label{eqn:gossip_iter_rel1_const}
  \|x^{k+1}_i - x^*_i\|^2 
 & \leq \|x_i^k-\alpha_{i} F_i(x_i^k,N\hat v^k_i) - x^k_i - x^*_i 
 + \alpha_{i} F_i(x^*_i,\bar x^*)\|^2 \cr
 & = \|x_i^k-x^*_i \|^2 + \alpha_{i}^2\| F_i(x_i^k,N\hat v^k_i) -F_i(x^*_i,\bar x^*) \|^2 \cr
  &- 2\alpha_{i} (F_i(x_i^k,N\hat v^k_i) -F_i(x^*_i,\bar x^*))^T(x^k_i-x^*_i).
\end{align}
By using $(a+b)^2\le 2a^2 + 2b^2$ and \us{Lemma}~\ref{cor:fibound}
we can see that
\[\| F_i(x_i^k,N\hat v^k_i) -F_i(x^*_i,\bar x^*) \|^2\le 4C^2.\]
We now approximate the inner product term by adding and subtracting 
$F_i(x^k_i,Ny^k)$ and {using} $y^k = \sum_{i=1}^N x^k_i=\bar x^k$ 
(see Lemma~\ref{lemma:true_agg}), to obtain
\begin{align*}
 (F_i(x_i^k,N\hat v^k_i) -F_i(x^*_i,\bar x^*))^T(x^k_i-x^*_i) 
  \ge & -|(F_i(x_i^k,N\hat v^k_i) -F_i(x^k_i,Ny^k))^T(x^k_i-x^*_i)| \\
 & + (F_i(x^k_i,\bar x^k) -F_i(x^*_i,\bar x^*))^T(x^k_i-x^*_i).
\end{align*}
By the Lipshitz property of the mapping $F_i$ in Assumption~\ref{assump:Lipschitz},
we have
\[|(F_i(x_i^k,N\hat v^k_i) -F_i(x^k_i,Ny^k))^T(x^k_i-x^*_i)| \le 
\bar L_iN\|\hat v^k_i-y^k\|\,\|x^k_i-x^*_i\|\le 
\bar L_iNM\|\hat v^k_i-y^k\|,\]
where $M\ge \max_{x_i,z_i}\|x_i-z_i\|$ for all $i$, which exists 
by compactness of each $K_i$.
Upon combining the preceding estimates with~\eqref{eqn:gossip_iter_rel1_const},
we obtain
\begin{align}
\label{eqn:gossip_iter_rel11}
  \|x^{k+1}_i - x^*_i\|^2 
 \leq &\|x_i^k-x^*_i \|^2 + 4\alpha_{i}^2 C^2
 +2\alpha_i\bar L_iNM \|\hat v^k_i-y^k\|\cr
  &- 2\alpha_{i} (F_i(x_i^k,\bar x^k) -F_i(x^*_i,\bar x^*))^T(x^k_i-x^*_i).
\end{align}

Now, we work with the last term in~\eqref{eqn:gossip_iter_rel11}, 
by letting $\alpha_{\min} = \min_{i} \alpha_i$, and
by adding and subtracting 
$2\alpha_{\min} (F_i(x_i^k,\bar x^k) -F_i(x^*_i,\bar x^*))^T(x^k_i-x^*_i)$,
we can see that
\begin{align}
\label{eqn:gossip_iter_rel12}
  \|x^{k+1}_i - x^*_i\|^2 
 \leq &\|x_i^k-x^*_i \|^2 + 4\alpha_{i}^2 C^2 +2\alpha_i\bar L_iNM \|\hat v^k_i-y^k\|\cr
  &+ 2(\alpha_i-\alpha_{\min})
  |(F_i(x_i^k,\bar x^k) -F_i(x^*_i,\bar x^*))^T(x^k_i-x^*_i)|\cr
  &- 2\alpha_{\min} (F_i(x_i^k,\bar x^k) -F_i(x^*_i,\bar x^*))^T(x^k_i-x^*_i).
\end{align}
By using the \jk{Cauchy-Schwarz} inequality and the Lipschitz property of $F_i$
given in Assumption~\ref{assump:Lipschitz_more}, we obtain
\begin{align}\label{eqn:estimf}
|(F_i(x_i^k,\bar x^k) -F_i(x^*_i,\bar x^*))^T(x^k_i-x^*_i)|\le L_i\|x^k_i-x^*_i\|^2.\end{align}
Further, by letting $\alpha_{\max} = \max_{i} \alpha_i$,
from~\eqref{eqn:gossip_iter_rel12} and~\eqref{eqn:estimf} by collecting 
the common terms we have for $i\in\{I^k,J^{k}\}$,
\begin{align*}  \|x^{k+1}_i - x^*_i\|^2 
 \leq &\left(1+2L_i(\alpha_{\max}-\alpha_{\min})\right)\|x_i^k-x^*_i \|^2 
 + 4\alpha_{i}^2 C^2+2\alpha_i\bar L_iNM \|\hat v^k_i-y^k\|\cr
  &- 2\alpha_{\min} (F_i(x_i^k,\bar x^k) -F_i(x^*_i,\bar x^*))^T(x^k_i-x^*_i).
\end{align*}
The fact that $x^{k+1}_i=x^k_i$ when $i \not \in \{I^k,J^{k}\}$ implies that 
$\|x^{k+1}_i - x^*_i\|^2 =  \|x_i^k-x^*_i \|^2$ for $i \not \in \{I^k,J^{k}\}$.
Next, we take the expectation in~\eqref{eqn:gossip_iter_last}, whereby we combine 
the preceding two cases and take into account that agent $i$ updates 
with probability $p_i$, and obtain
for all $i \in \mathcal{N}$,
\begin{align*}
  \mathbb{E}[\|x^{k+1}_i - x^*_i\|^2 &\mid\mathcal{F}_k]
 \leq \left(1+2p_i L_i(\alpha_{\max}-\alpha_{\min})\right)\|x_i^k-x^*_i \|^2 
 + 4p_i\alpha_{i}^2 C^2\cr
 &+2p_i\alpha_i\bar L_iNM \,\mathbb{E}[\|\hat v^k_i-y^k\|\mid\mathcal{F}_k]
  - 2p_i\alpha_{\min} (F_i(x_i^k,\bar x^k) -F_i(x^*_i,\bar x^*))^T(x^k_i-x^*_i).
\end{align*}
Since $\hat v_i^k$ is a convex combination of $v_j^k$, $j\in\Nscr$ and
 the norm is a convex function, we have 
 $\|\hat v^k_i-y^k\|\le \sum_{j=1}^N [W(k)]_{ij}\|v^k_j-y^k\|$. Thus,
$ \mathbb{E}[\|\hat v^k_i-y^k\|\mid \mathcal{F}_k]
  \le \sum_{j=1}^N \mathbb{E}[W(k)]_{ij}\|v^k_j-y^k\|.$
By using the preceding relation, $\min_i p_i=p_{\min}$, $p_{\max}=\max_i p_i$, 
and $\alpha_i\le\alpha_{\max}$,
we arrive at the following relation for all $i \in \mathcal{N}$,
\begin{align}
\label{eqn:gossip_iter_last}
 &\mathbb{E}[ \|x^{k+1}_i - x^*_i\|^2 \mid\mathcal{F}_k]
 \leq (1+2p_{\max} \max_i L_i(\alpha_{\max}-\alpha_{\min}) )\|x_i^k-x^*_i \|^2 
 + 4p_{\max}\alpha_{\max}^2 C^2 \cr
 &+2p_{\max}\alpha_{\max} B \sum_{j=1}^N \mathbb{E}[W(k)]_{ij}\|v^k_j-y^k\|
 - 2p_{\min}\alpha_{\min} (F_i(x_i^k,\bar x^k) -F_i(x^*_i,\bar x^*))^T(x^k_i-x^*_i),
\end{align}
with $B=(\max_i \bar L_i)NM$.

Summing the relations in~\eqref{eqn:gossip_iter_last} over all $i=1,\ldots,N$, 
recalling that $F_i$, $i=1,\ldots,N$ are coordinate maps for the map $F$
(see~\eqref{eqn:svinot}), which in turn defines the mapping $\phi$ 
through~\eqref{eqn:phimap}, we further obtain
\begin{align*}
 &\mathbb{E}[ \|x^{k+1} - x^*\|^2 \mid\mathcal{F}_k]
 \leq (1+2p_{\max} \max_i L_i(\alpha_{\max}-\alpha_{\min}) )\|x^k-x^*\|^2 
 + 4p_{\max}\alpha_{\max}^2 C^2N\cr
 &+2p_{\max}\alpha_{\max} B \sum_{i=1}^N\sum_{j=1}^N \mathbb{E}[W(k)]_{ij}\|v^k_j-y^k\|
 - 2p_{\min}\alpha_{\min} (\phi(x^k) - \phi(x^*))^T(x^k - x^*),
\end{align*}  
The matrix $\mathbb{E}[W(k)]$ is doubly stochastic, so we have
\[\sum_{i=1}^N\sum_{j=1}^N\mathbb{E}[W(k)]_{ij}\|v^k_j-y^k\|
=\sum_{j=1}^N\sum_{i=1}^N\mathbb{E}[W(k)]_{ij}\|v^k_j-y^k\|
=\sum_{j=1}^N\|v^k_j-y^k\|.\]
Using this relation and the strong monotonicity of the mapping $\phi$ 
with a constant $\mu,$  gathering the common terms, and taking the total expectation, we 
obtain for all $k\ge0$,
\begin{align}
\label{eqn:gossip_iter_ende}
 \mathbb{E}[ \|x^{k+1} - x^*\|^2]
 \leq q \mathbb{E}[\|x^k-x^*\|^2] + 4p_{\max}\alpha_{\max}^2 C^2N
 +2p_{\max}\alpha_{\max} B \sum_{j=1}^N \mathbb{E}[\|v^k_j-y^k\|],
\end{align} 
where $q=1-2\mu p_{\min}\alpha_{\min}+2p_{\max} \max_i L_i(\alpha_{\max}-\alpha_{\min})$.
Note that by the condition 
\[0<1-2\mu p_{\min}\alpha_{\min}+2p_{\max} \max_i
L_i(\alpha_{\max}-\alpha_{\min})<1,\]
we have $0<q<1$. We further have $\{x^k\}\subseteq K$ for a compact set $K$, so
the limit superior of $\mathbb{E}[ \|x^{k} - x^*\|^2]$ exists.
Thus, by taking the limit as $k\to\infty$ in relation~\eqref{eqn:gossip_iter_ende}
and using Lemma~\ref{lem:conv_estimate_const}, we obtain 
\begin{align*}
 \limsup_{k\to\infty}\mathbb{E}[ \|x^{k+1} - x^*\|^2]
 \leq q 
 \limsup_{k\to\infty}\mathbb{E}[\|x^k-x^*\|^2] + 4p_{\max}\alpha_{\max}^2 C^2N
 +2p_{\max}\alpha_{\max} B \frac{\sqrt{2nN}\, C\alpha_{\max}}{1-\sqrt{\lambda}},
\end{align*}
which implies the stated result.
\end{proof}

We have few comments on the result of Proposition~\ref{prop:gs_convergence_const}, as follows.
The error bound depends on the dimension $n$ of the decision variables, the number $N$ of players, the frequency with which players update their decisions (captioned by $p_{\min}$ and $p_{\max}$), and 
the network properties including the connectivity time bound $B$ and the ability to propagate the information
(captured by the value $1-\sqrt{\lambda}$).
When the network parameters $B$ and $\lambda$, and the players' update probabilities 
($p_{\min}$ and $p_{\max}$) do not depend on $N$, the error bound grows linearly with 
the number $N$ of players.

As a special case, consider the case when the agents employ an equal stepsize, 
i.e., $\alpha_{\min} =\alpha_{\max}=\alpha$ and  $\alpha$ satisfies the following condition
$0<\alpha < \frac{1}{2\mu p_{\min}}$.
Then, the result of Proposition~\ref{prop:gs_convergence_const} reduces to
\[\limsup_{k\to\infty}\mathbb{E}[ \|x^{k} - x^*\|^2]
 \le \frac{p_{\max}\alpha \left( 2 C^2N + BC \frac{\sqrt{2nN}} {1-\sqrt{\lambda}} \right) }
 {\mu p_{\min}}.\]
As another special case, consider the case when all players have equal probabilities of updating, i.e., 
$p_{\min}=p_{\max}=p.$ Then, 
we have the following result:
\[\limsup_{k\to\infty}\mathbb{E}[ \|x^{k} - x^*\|^2]
 \le \frac{ \alpha_{\max}^2 \left( 2C^2N + B C\frac{\sqrt{2nN}}{1-\sqrt{\lambda}} \right) }
 {\mu \alpha_{\min}-(\max_i L_i)(\alpha_{\max}-\alpha_{\min})},\]
When all players have equal probabilities of updating and all use equal stepsizes, 
i.e., $p_{\min}=p_{\max}=p$ and $\alpha_{\min} =\alpha_{\max}=\alpha,$ then the condition
of Proposition~\ref{prop:gs_convergence_const} reduces to $0<\alpha < \frac{1}{2\mu p},$
and the bound further simplifies to:
\[\limsup_{k\to\infty}\mathbb{E}[ \|x^{k} - x^*\|^2]
 \le \frac{\alpha} {\mu}\left( 2C^2N + B C\frac{\sqrt{2nN}}{1-\sqrt{\lambda}} \right).\]

\section{\us{Generalizations and Extensions}} 
\us{In prior sections, we have developed two algorithms for addressing a
class of Nash games. To recap, our prescribed class of
equilibrium computation schemes may accommodate a specific sublcass of
noncooperative $N-$person Nash games, qualified as {\em aggregative}.
Specifically, in such games, given $\bar x_{-i}$, the $i$th player
solves the deterministic convex program given by \eqref{eqn:problem}.
Unlike in much of prior work, players cannot observe $\bar x_{-i}$ but
may learn it through the exchange of information with their local
neighbors based on an underlying graph. This underlying graph may either
be time-varying with a connectivity requirement or be fixed. In the case
of the former, we present a synchronous scheme that mandates that every
agent synchronizes its updates and information exchanges while in the
latter case, we develop a an asynchronous gossip protocol for communication.
Under a strict monotonicity and suitably defined Lipschitzian
requirements on the map $F$ and compactness of the set $K$, we show that
both algorithms produce sequences that are guaranteed to converge to the
unique equilibrium in an almost sure sense. More succinctly, under
appropriate communication requirements,  any {\em
deterministic convex aggregative Nash game} may be addressed through such
synchronous/asynchronous techniques as long as the requirements on the
map and strategy sets hold. } \\

\us{Naturally, one  may rightly question whether the presented
schemes (and their variants) can accommodate weakening some of the
assumptions, both on the map and more generally on the model. Motivated by this concern, in Section~\ref{sec:gen}, we begin by
discussing how extensions of the algorithm can allow for accommodating
weaker assumptions on the problem setting. Subsequently, in
Section~\ref{sec:extension}, we discuss how the very nature of the
coupling across agents can also be generalized.}

\subsection{Weakening assumptions on problem parameters}\label{sec:gen}
We consider three extensions to our prescribed class of games:

\noindent {\em Strict monotonicity:} \us{The a.s. convergence theory for
both the synchronous and asynchronous schemes  rely on the strict
monotonicity of the map as asserted by Assumption \ref{ass-strict-mon}.
There are several avenues for weakening such a requirement. For
instance, one approach relies on using a regularized variant of
\eqref{eqn:x_algo}, given by 
\begin{align}\label{eqn:x_algo_reg}
 x_i^{k+1} & =\Pi_{K_i}[x_i^k-\alpha_k (F_i(x_i^k,N\hat v^k_i) +
 \epsilon_{k,i}x_i^k)], \\
\label{eqn:estimate_update_reg}
 v_i^{k+1} & = {\hat v_i^k} + x_i^{k+1}-x_i^k,
\end{align}
where $\{\epsilon_{k,i}\}$ denotes a sequence that diminishes to zero at
a prescribed rate. Such an approach has been employed in prior work
and allow for the solution of both deterministic~\cite{kannan10online}
and stochastic~\cite{koshal13regularized} monotone variational inequality
problems. An alternative approach may lie in the usage of an
extragradient framework that requires taking two, rather than one,
gradient step. Via such approaches, it
			  has been shown~\cite{juditsky2011} that the gap function
			  associated with a monotone stochastic variational
			  inequality problem  tends to
			  zero in mean. We believe
distributed counterparts of extragradient schemes hold significant
and represent a generalization of the schemes presented in this
paper.} \\

\noindent {\em Lipschitz continuity:} \us{A second assumption employed in
deriving convergence and rate statements is the (uniform) Lipschitzian
assumption as articulated in Assumption~\ref{assump:Lipschitz}. We
believe that there are at least two avenues that can be adopted in
weakening this requirement. First, there has been significant recent
work that integrates the use of local or randomized smoothing (also
		called Steklov-Sobolev smoothing) to address stochastic
variational inequalities in which the maps are not necessarily Lipschitz
continuous (cf.~\cite{yousefian13regularized}). It may well be possible
to extend such techniques to address settings where the uniformly
Lipschitzian assumption does not hold. An alternate approach may lie in
developing convergence in mean of the gap function, as adopted in
~\cite{juditsky2011} where the either Lipschitz continuity or
boundedness of the map is necessary.}\\

\noindent {\em Stochastic payoff functions:} \us{Presently, the main source of
uncertainty arises either from the evolution of the connectivity graph
(synchronous scheme) or the randomness in the choice of players that
communicate as per the gossip protocol (asynchronous) scheme. Yet, the
player objectives could also be expectation-valued. Consequently, the
equilibrium conditions are given by a stochastic variational inequality.
In such instances, one may articulate suitably defined distributed stochastic approximation
counterparts of \eqref{eqn:x_algo} to cope with such a challenge
(cf.~\cite{koshal13regularized}). We
believe convergence analysis of such schemes, while more complicated,
is likely to carry through under suitable assumptions.}

\subsection{\us{Extensions to model}}\label{sec:extension}
It may have been observed that the proposed developments in the
	earlier two sections required that the agent decisions be of the same dimension. In this \us{subsection}, we extend the realm of~\eqref{eqn:problem} and 
generalize the algorithms presented in section~\ref{sec:synchron} and 
section~\ref{sec:asynchron}.  To this end, consider the following
aggregative game
\begin{align}\label{eqn:problem_general}
\mbox{minimize} & \qquad f_i\left(x_i,\sum_{i=1}^N h_i(x_i)\right) \cr
\mbox{subject to} & \qquad x_i \in K_i,
\end{align}
where $K_i \subseteq \mathbb{R}^{n_i},$ $h_i :  K_i \to \mathbb{R}^n$. 
The mappings \jk{$f_i$} and $h_i$ are considered to be information private to player $i$. Such an 
extension allows players decisions to
	have different 
dimensionality. To recover the problem articulated by~\eqref{eqn:problem}, 
we set $h_i(x_i) = x_i$ with
$K_i \subseteq \mathbb{R}^n$ for all $i$.
Next, we discuss the generalization of the proposed distributed
algorithms to solve problem~\eqref{eqn:problem_general}.\\

\noindent {\em Synchronous Algorithm:}
To make the synchronous algorithm suitable for the generalized 
problem in~\eqref{eqn:problem_general}, the mixing step in~\eqref{eqn:estimate_mixing}
remains the same, but with a different initial condition. Namely, the
mixing in~\eqref{eqn:estimate_mixing} is initiated with
\begin{equation}\label{eqn:syn_initial_gen}
v_i^0 = h_i(x_i^0) \qquad \mbox{ for all } i=1,\ldots,N,
\end{equation}
where $x_i^0 \in K_i$ are initial players' decisions.
The iterate update of~\eqref{eqn:x_algo} and 
update \us{of the average estimate}~\eqref{eqn:estimate_update_reg} \us{are
	modified, leading to the following}:
\begin{align}\label{eqn:x_algo_gen}
 x_i^{k+1} & =\Pi_{K_i}[x_i^k-\alpha_k F_i(x_i^k,N\hat v^k_i)],\\
\label{eqn:estimate_update_gen}
 v_i^{k+1}&=\hat v_i^k + h_i(x_i^{k+1})-h_i(x_i^k),
\end{align}
where $\alpha_k$ is the stepsize, the mapping $F_i$ is given by
\begin{equation}\label{eqn:F_general}
F_i\left( x_i,\sum_{i=1}^Nh_i(x_i) \right) = \nabla_{x_i} f_i\left( x_i,\sum_{i=1}^Nh_i(x_i)\right),
\end{equation}
and $N\hat v^k_i$ in~\eqref{eqn:x_algo_gen} is an estimate \us{of} the true value $\sum_{i=1}^Nh_i(x_i)$.
\us{In the context of the extended synchronous algorithm} in the preceding discussion, we have the following result.
\begin{proposition}
Let Assumptions~\ref{assump:set_func}--\ref{assump:step_diffusion} hold for the mapping 
$\phi(x)=(\phi_1(x),\ldots,\phi_N(x))^T$ with coordinates 
$\phi_i(x)=\nabla_{x_i} f_i\left( x_i,\sum_{i=1}^Nh_i(x_i)\right)$ and
$x=(x_1^T,\ldots,x_N^T)^T$.
Then, the sequence $\{x^k\}$ generated by the method~\eqref{eqn:x_algo_gen}--\eqref{eqn:estimate_update_gen}
converges to the  (unique) solution $x^*$ of the game in~\eqref{eqn:problem_general}.
\end{proposition}
\begin{proof}
The proof mimics the proof of Proposition~\ref{prop:diffusion_convergence}.
\end{proof}

\noindent {\em Asynchronous Algorithm:}
\us{We now discuss how the gossip algorithm in
	section~\ref{sec:asynchron}} \us{may} be modified. \us{While the} estimate mixing in~\eqref{eqn:gossip_mixing} remains unchanged, the initial condition 
is replaced \us{by} one given in~\eqref{eqn:syn_initial_gen}.
The iterate update of~\eqref{eqn:gossip_x_algo} and average estimate update of~\eqref{eqn:gossip_estimate_update} are modified, as follows:
\begin{align}\label{eqn:gossip_x_algo_gen} 
x_i^{k+1}&:=(\Pi_{K_i}[x_i^k-\alpha_{k,i} F_i(x_i^k,N\hat v^k_i)] - 
x^k_i )\mathbbm{1}_{\{i \in \{I^k,J^k\}\}} + x^k_i,\\
\label{eqn:gossip_estimate_update_gen}
 v_i^{k+1} &:=\hat v_i^k+h_i(x_i^{k+1})-h_i(x_i^k),
\end{align}
where $\alpha_{k,i}$ is the stepsize for user $i$ and the mapping 
$F_i$ is as defined in~\eqref{eqn:F_general}.
The following result establishes the convergence of the 
extended asynchronous algorithm.
\begin{proposition}
Let Assumptions~\ref{assump:set_func}--\ref{assump:Lipschitz} and Assumption~\ref{asum:graph} hold. 
Then, the sequence $\{x^k\}$ generated by the method
\eqref{eqn:gossip_x_algo_gen}--\eqref{eqn:gossip_estimate_update_gen} with stepsize $\alpha_{k,i}=\frac{1}
{\Gamma_k(i)}$ converges to the (unique) $x^*$ of the game almost surely.
\end{proposition}
\begin{proof}
With the initial condition specified by~\eqref{eqn:syn_initial_gen}, 
 the  proof \us{follows in a fashion similar to that of}
Proposition~\ref{prop:gs_convergence}.
\end{proof}


\section{Numerics}\label{sec:numerics}
In this section, we examine the performance of the proposed
algorithms on a class of Nash-Cournot games. Such games represent an
instance of aggregative Nash games and in
section~\ref{sec:game_numeric}, we describe the player payoffs and
strategy sets as well as verify that they satisfy the necessary
assumptions. In section~\ref{sec:diffusion_numeric}, we discuss the
synchronous setting and present the results arising from applying our
algorithms. In section~\ref{sec:gossip_numeric}, we turn our attention
to asynchronous regime where we present our numerical experience of
applying the gossip algorithm.

\subsection{Nash-Cournot Game}\label{sec:game_numeric}
We consider a networked Nash-Cournot games which is possibly 
amongst the best known examples of an aggregative game. Specifically, the aggregate in 
such games is the total sales which is the sum of production over all the players. 
The market price is set in accord with an inverse demand 
function which depends on the aggregate of the network. A formal
description of such a game over a network is provided in
Example~\ref{ex:nash_example}. 
{Before proceeding to describe our
experimental setup, we show that Nash-Cournot games do indeed satisfy
Assumptions~\ref{assump:Lipschitz} 
and~\ref{assump:Lipschitz_more}, respectively, under some mild conditions
on the cost and price functions. It is worth pointing that we have used 
Assumption~\ref{assump:Lipschitz_more} only for the error bound results for
the asynchronous algorithm with a constant stepsize.
}

{
In the sequel, within the context of Example~\ref{ex:nash_example}, 
we let $x_{il} = (g_{il},s_{il})$ for all $l=1,\ldots,\mathcal{L}$,  
$x_i = (x_{i1}, \hdots, x_{i\mathcal{L}})$ and $x=(x_1, \ldots, x_N)^T.$ 
Further, we define coordinate maps $F_i(x_i,u)$, as follows:
\begin{equation}
\label{eqn:numeric_F}
F_i(x_i,u) =  \pmat{ F_{i1}(x_{i1},u_1) \\
					\vdots \\
					F_{i\mathcal{L}}(x_{i\mathcal{L}},u_{\mathcal{L}})}, 
					\qquad F_{il}(x_{il},u_l) = \pmat{ c'_{il}(g_{il}) \\
									-p_l(u_l) - p_l'(u_l)s_{il}},
\end{equation}
where the prime \jk{denotes} the first derivative.
We let
$F(x, u)=(F_1(x_1, u)^T, \ldots, F_N(x_N, u)^T)^T$, and
$K_i$ denote the constraint set on player $i$ decision,
$x_i$, as given in Example~\ref{ex:nash_example}.}

{We note that the Nash-Cournot game under the consideration
satisfies Assumption~\ref{assump:set_func} as long as the cost functions $c_{il}$ are convex and 
the price functions $p_l(u_l)$ are concave for all $i$ and $l$.
Furthermore, the strict convexity condition of Assumption~\ref{ass-strict-mon} is satisfied
when, for example, all price functions $p_l$ are strictly concave. This can be seen by observing that 
\begin{align*} (F(x,u)-F(\tilde x,u))^T(x-\tilde x) & = 
\sum_{l=1}^{\mathcal{L}} \sum_{i=1}^N
(c_{il}'(g_{il})-c_{il}'(\tilde g_{il}))(g_{il} - \tilde g_{il}) - 
\sum_{l=1}^{\mathcal{L}} \sum_{i=1}^N p_{l}'(u_l) (s_{il} - \tilde s_{il})^2.
\end{align*}
}

{Next, we show that the Lipschitzian
requirements on the maps $F_i$ of Assumption~\ref{assump:Lipschitz} holds 
under some mild assumptions on the cost and price
functions in Nash-Cournot games, as shown next.
}

\begin{lemma}
\label{lemma:cournot-lip}
{Consider the Nash-Cournot game described in Example~\ref{ex:nash_example}. Suppose that 
each $p_l(u_l)$ is concave and has Lipschitz continuous derivatives with a constant $M_l$ (over a 
coordinate projection of $\bar K$ on the $l$th coordinate axis).
Then, the following relation holds:
\[ \|F_i(x_i,u)-F_i(x_i,z)\| 
\leq \sqrt{2}\sqrt{\sum_{l=1}^\mathcal{L}(C_l^2+M_l^2 {\rm cap}_{il}^2)}\,\|u - z\|
\qquad\hbox{for all }u,z\in\bar K.\]
	}
\end{lemma}
\begin{proof}
\an{
This result follows directly from the definition of the coordinate maps $F_{il}(x_{il},u)$ and
recalling that $x_{il} = (g_{il},s_{il})$. In particular, for each $i,\ell$ we have 
\begin{align*}
\|F_{il}(x_{il},u_l)-F_{il}(x_{il},z_l)\| 
& = \sqrt{|c_{il}'(g_{il}) - c_{il}'(g_{il})|^2
+|p_l(u_l) + p'_l(u_l)s_{il}-p_l(z_l)-p'_l(z_l)s_{il}|^2} \cr
& = \sqrt{|(p_l(u_l) - p_l(z_l))+ (p'_l(u_l) - p'_l(z_l))s_{il}|^2} \cr
&\le \sqrt{2}\sqrt{|p_l(u_l) - p_l(z_l)|^2+ |p'_l(u_l) - p'_l(z_l)|^2s_{il}^2}, 
\end{align*}
where the inequality follows from $(a+b)^2\le 2(a^2+b^2)$.
Since $\bar K$ is compact and each $p_l$ has continuous derivatives, it follows that 
there exists a constant $C_l$ for every $l$ such that 
\[|p'_l(u_l)|\le C_l\qquad\hbox{for all $u_l$ with $u=(u_1,\ldots,u_\mathcal{L})^T\in \bar K$}.\]
Then, by using concavity of $p_l$, we can see that 
$|p_l(u_l) - p_l(z_l)|\le C_l|u_l - z_l|$ implying that 
\begin{align*}
\|F_{il}(x_{il},u_l)-F_{il}(x_{il},z_l)\| 
&\le \sqrt{2}\sqrt{C_l^2|u_l - z_l|^2 + |p'_l(u_l) - p'_l(z_l)|^2s_{il}^2}\cr
&\le \sqrt{2}\sqrt{(C_l^2+M_l^2 s_{il}^2)}|u_l - z_l|,
\end{align*}
where the last inequality is obtained by using the Lipschitz property of the  derivative $p'_l(u_l)$.
From the structure of constraints we have $s_{il}\le {\rm cap}_{il}$ yielding
\begin{align*}
\|F_{il}(x_{il},u_l)-F_{il}(x_{il},z_l)\| 
&\le \sqrt{2}\sqrt{C_l^2+M_l^2 {\rm cap}_{il}^2}\,|u_l - z_l|\qquad\hbox{for all } l.
\end{align*}
Further, by using H\"older's inequality, and 
recalling that $x_i=(x_{i1},\ldots,x_{i\mathcal{L}})$ and $u=(u_1,\ldots,u_{\mathcal{L}})$,
from the preceding relation we obtain
\begin{align*}
\|F_i(x_i, u) - F_i(x_i,z)\| 
& = \sqrt{\sum_{l=1}^\mathcal{L} \|F_{il}(x_{il},u_l)-F_{il}(x_{il},z_l)\|^2 }
 \leq \sqrt{2}\sqrt{\sum_{l=1}^\mathcal{L}(C_l^2+M_l^2 {\rm cap}_{il}^2)}\,\|u - z\|.
\end{align*}
}
\end{proof}

\us{We now show that $F_i(x,u)$ is Lipschitz continuous in $x$ for every
	$u$.}
\begin{lemma}
\label{lemma:cournot-lip2}
{Consider the Nash-Cournot game described in Example~\ref{ex:nash_example}. 
Suppose that each $c_{il}'$ is Lipschitz continuous with a constant $L_{il}$ and 
$|p'_l(u)| \leq \bar p_l$  for some scalar $\bar p_l$ and for all $u\in\bar K$. 
Then, the following relation holds for all $i$,
\[\|F_{i}(x_{i},u)-F_{i}(\tilde x_{i}, u)\| 
\leq \sqrt{\sum_{l =1}^{\mathcal{L}}\left( L_{il}^2 + \bar p_l^2\right)}\,
\|x_{i} - \tilde x_{i}\|\qquad\hbox{for all }x_i,\tilde x_i\in K_i.\]
}
\end{lemma}
\begin{proof}
\an{
First, we note that for each $i,l$,
\begin{align*}
\|F_{il}(x_{il},u_l) - F_{il}(\tilde x_{il}, u_l)\| 
& = \sqrt{| c_{il}'(g_{il}) - c_{il}'(\tilde g_{il})|^2+|p'_l(u_l)(s_{il}-\tilde s_{il})|^2} \cr
& \le \sqrt{L_{il}^2 |g_{il} - \tilde g_{il}|^2+\bar p_l^2 |s_{il}-\tilde s_{il}|^2}.
\end{align*}
Recalling our notation $x_{il} = (g_{il},s_{il})$ and $\tilde x_{il} = (\tilde g_{il},\tilde s_{il})$, and using
H\"older's inequality, we find  that 
\[\|F_{il}(x_{il},u_l)-F_{il}(\tilde x_{il}, u_l)\| 
\leq \sqrt{ L_{il}^2 + \bar p_l^2}\, \|x_{il} - \tilde x_{il}\|.\]
Further, recalling that $x_i=(x_{i1},\ldots, x_{i\mathcal{L}})$, 
$\tilde x_i=(\tilde x_{i1},\ldots, \tilde x_{i\mathcal{L}}),$ and $u=(u_1,\ldots, u_{\mathcal{L}})$,
the desired result follows from the preceding relation by using H\"olders inequality.
}
\end{proof}

In our numerical study, we consider a Nash-Cournot game 
	played over ten locations, i.e. $\mathcal{L} =10$, in
	which all players have {cost functions of a similar structure} and the $i$th player's
		optimization problem \us{may} be expressed as
\begin{align}\label{eqn:game_numeric}
\mbox{minimize} & \qquad \sum_{l =1}^{10} \left( c_{il} (g_{il}) - p_l(\bar s_l)
		s_{il}\right)\cr
\mbox{subject to} & \qquad \sum_{l=1}^{10} g_{il} = \sum_{l=1}^{10} s_{il},
	\notag \\& \qquad g_{il}, s_{il} \geq 0,\quad g_{il} \leq \mathrm{cap}_{il}, \qquad l = 1,
	\hdots, {10}, 
\end{align}
where $g_{il}$ and $s_{il}$ denote player $i$'s production and sales at location $l,$ respectively, 
and $\bar s_l$ denotes the aggregate of all the players' decisions 
($\bar s_l = \sum_{i=1}^N s_{il}$) at location $l$. 
\an{The function $c_{il}(g_{il})$ denotes the 
cost of production for $i$th player at location $l$ and has the following form:
\[c_{il}(g_{il}) =  a_{il}g_{il} + b_{il} g_{il}^2,\]
where $a_{il}$ and $b_{il}$ are scaling parameters for agent $i$. 
In our experiments, we draw $a_{il}$ and $b_{il}$ from 
a uniform distribution and fix them over the course of the entire 
simulation. More precisely, for $i=1,\ldots, N,$ and $l=1, \ldots, 10,$ we have $a_{il}
\sim U(2,12)$ and $b_{il} \sim U(2,3),$ where $U(t,\tau)$ denotes the
uniform distribution over an interval $[t,\tau]$ with 
$t< \tau.$ 
The term $p_l(\bar s_l)$ captures the inverse demand function and takes the following form:
\[p_l(\bar s_l) = d_l - \bar s_l,\]
where $d_l$ is a parameter for location $l$. The parameters $d_l$ are also drawn randomly with a uniform distribution, $d_l \sim U(90, 100)$ for all $l=1,\ldots, 10.$
Furthermore, we use $cap_{il} = 500$ for all $i = 1, \ldots, N$ and for all $l=1,\ldots, 10.$ 
The affine price function gives rise to a strongly
	monotone map $\phi=F(x,\bar x)$, which together with the compactness of the sets $K_i$, implies that
	this game has a unique Nash equilibrium.
	} Note that in our setup, we have $cap_{il} > d_l$ indicating that at a particular location, a player may 
	produce more than the overall demand at that particular location. Such a scenario can arise when it might 
	be more efficient to produce at a location to meet the demand(s) of another location(s) assuming the transportation costs are zero.
\subsection{Synchronous Algorithm}\label{sec:diffusion_numeric}
In this section, we investigate the performance of synchronous 
algorithm of section~\ref{sec:synchron} for the computation of 
the equilibrium of aggregative game~\eqref{eqn:game_numeric}.
\an{We begin by describing our setting for the connectivity graph of 
the network of players, where each player is seen as a node in a graph. 
At each iteration $k,$ we generate a symmetric $N\times N$ adjacency matrix $A$ such 
that the underlying graph is connected. The entries of 
$A$ are generated by performing the following steps:
\begin{itemize}
\item[(0)] Let $I$ denote the set of nodes that have already 
been generated;
\item[(1)] For each newly generated node $j$, select a node 
randomly $i \in I$ to establish an edge $\{i,j\}$ and set 
$[A]_{ij}=[A]_{ji}=1$;
\item[(2)] Repeat step 1 until $I=\{1,\ldots,N\}$.
\end{itemize}
}

Given such an adjacency matrix $A,$ we define a doubly stochastic 
symmetric weight matrix $W$ such that
\[
 [W]_{ij} = 
 \left\{
  \begin{array}{ll}
   0        & \mbox{if } A_{ij} = 0 \\
   \delta & \mbox{if $A_{ij} = 1$ and $i\ne j$} \\
   1 - \delta d(i) & \mbox{if } i =j,
  \end{array}
 \right.
\]
where $d(i)$ represents the number of players communicating with player $i,$ and 
\[\delta = \frac{0.5}{\max_i \{d(i)\}} .\]
Using the adjacency matrix $A$ and the weight matrix $W$, players update their decision
and their estimate of the average using~\eqref{eqn:estimate_mixing}--\eqref{eqn:estimate_update}.
The stepsize rule for agent update is as follows:
\[\alpha_{k,i} = \frac{1}{k} \quad \mbox{for all } i=1,\ldots, N.\]
The algorithm is initiated {at a random starting point, and terminated after 
a fixed number of iterations, denoted by $\tilde k,$ for each
sample path. We use a set of 50 sample paths for each simulation setting,
and we report the empirical mean of the sample errors, defined as:}
\begin{align}
\label{eqn:def-err_ell}
\us{
\mbox{error}_{\tilde k} 
=
\frac{\displaystyle \max_{i \in \{1, \ldots,N\}} \max_{l \in \{1, \ldots,
	10\}} \left\{|g_{il}^{\tilde k}-g_{il}^*|, |s_{il}^{\tilde k}-
s_{il}^*|\right\}}{\displaystyle \max_{i\in \{1, \ldots,N\}} \max_{l\in
	\{1, \ldots,
	10\}} \left\{|g_{il}^*|, |s_{il}^*|\right\}},
}
\end{align}
where $g^*_{il}$ and $s^*_{il}$ are the decisions of agent $i$ at the Nash 
equilibrium. The Nash equilibrium decisions $g^*_{il}$ and $s^*_{il}$
are computed using a constant steplength gradient projection algorithm
assuming each agent has true information of the aggregate. Note that
such an algorithm is guaranteed to converge under the strict convexity
of the players' costs.

We investigate cases with 20 and 50 players in the network.
In Table~\ref{tab:diffusion_err} and Table~\ref{tab:diffusion_cf}, we
report the \us{empirical mean of the error}
and \jk{90\%} confidence interval attained for different levels of $\tilde{k}$,
\us{(the simulation length)}, respectively.
Some insights that can be drawn from the simulations are
provided next:
\begin{itemize}
\item Expectedly, as seen in Table~\ref{tab:diffusion_err} and Table~\ref{tab:diffusion_cf}, the 
\us{empirical mean of the error upon termination} and the width of the confidence interval 
\us{decrease} with increasing $\tilde k$ and  \us{increase}  with 
network size.
\item The impact of the time-varying nature of the connectivity graph is
explored by considering a static complete graph as a basis for
comparison. 
In Table~\ref{tab:diffusion_error_static} and
Table~\ref{tab:diffusion_conf_int_static}, we
report the mean error and the confidence interval when the network is static.
Under this setting, the agents have access to the true 
aggregate information throughout the run of the algorithm. 
Naturally, the performance of the algorithm on a static 
complete network is orders of magnitude better than that on a 
dynamic network. This deterioration in performance may be
interpreted as the price of information from the standpoint of
convergence.
\end{itemize}

\begin{table}[H]
\begin{minipage}[b]{0.47\linewidth}
\centering
\caption{Dynamic network: Mean error on termination vs network size for various thresholds}
\vspace{0.15cm}
\label{tab:diffusion_err}
\begin{tabular}{c|c|c}
\hline
$N$ &  $\tilde{k}=5e$3 &  $\tilde{k}=1e$4 \\
\hline
20	&  9.22$e$-5	&	3.66$e$-5 \\
\hline
50	&	 8.38$e$-2  &	2.65$e$-3 \\
\hline
\end{tabular}
\end{minipage}
\hspace{0.1cm}
\begin{minipage}[b]{0.47\linewidth}
\centering
\caption{Dynamic network: Width of \jk{90\%} confidence interval of mean error}
\vspace{0.15cm}
\label{tab:diffusion_cf}
\begin{tabular}{c|c|c}
\hline
$N$ &  $\tilde{k}=5e$3 &  $\tilde{k}=1e$4 \\
\hline
20	& 2.147$e$-4	& 9.33$e$-5	\\
\hline
50	& 1.24$e$-1	& 2.78$e$-2	\\
\hline
\end{tabular}
\end{minipage}
\end{table}

\begin{table}[H]
\begin{minipage}[b]{0.47\linewidth}
\centering
\caption{Static network: Mean error on termination}
\vspace{0.15cm}
\label{tab:diffusion_error_static}
\begin{tabular}{c|c|c}
\hline
$N$ &  $\tilde{k}=5e$3 &  $\tilde{k}=1e$4 \\
\hline
20	&  3.66$e$-5	&	3.66$e$-5 \\
\hline
50	&	 6.23$e$-5  &	6.23$e$-5 \\
\hline
\end{tabular}
\end{minipage}
\hspace{0.1cm}
\begin{minipage}[b]{0.47\linewidth}
\centering
\caption{Static network: Width of \jk{90\%} confidence interval of mean error}
\label{tab:diffusion_conf_int_static}
\vspace{0.15cm}
\begin{tabular}{c|c|c}
\hline
$N$ &  $\tilde{k}=5e$3 &  $\tilde{k}=1e$4 \\
\hline
20	& 6.99$e$-9	& 6.99$e$-9\\
\hline
50	& 1.61$e$-8	& 1.61$e$-8\\
\hline
\end{tabular}
\end{minipage}
\end{table}

\subsection{Asynchronous Algorithm}\label{sec:gossip_numeric}
We now demonstrate the performance of the asynchronous algorithm
of section~\ref{sec:asynchron}. We consider four 
instances of connectivity graphs which we describe next and, also, 
depict these graphs in Figure~\ref{fig:various_network}\footnote{The network topology shown in 
Figure~\ref{fig:various_network} is for demo purposes only. 
For instance, Figure~\ref{fig:various_network}(a) is an example of a cycle network with 5 players.}.
\begin{itemize}
\item {\em Cycle}: Every player has two neighbors; 
\item {\em Wheel}: There is one central player that is connected to every other
player;
\item {\em Grid}: Players on the vertex have two neighbors, players on the edge 
have three and everyone else has four neighbors. Each row in the 
grid consists of five players and there are $N/5$ rows where $N$
is the size of the network;
\item{\em Complete graph:} Every player has an edge
	connecting it to every other player. 
\end{itemize}
Note that, in each connectivity graph, players can only communicate with their immediate neighbors.

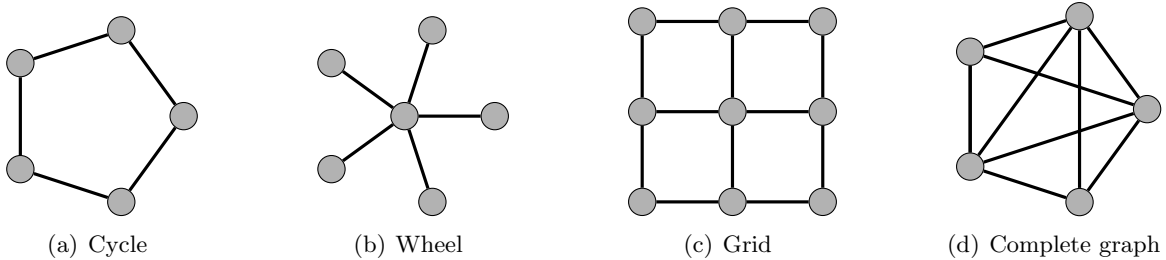
\begin{figure}[H]
\centering
\subfigure[Cycle]{
\begin{tikzpicture}[scale=1.2]
\tikzstyle{vertex}=[draw,shape=circle];
\tikzstyle{selected}=[vertex,fill=black];
\tikzstyle{neighbor}=[vertex,fill=black!30];
\tikzstyle{nonneighbor}=[vertex,fill=black!0];

\tikzstyle{edge} = [draw,thick,-]
\tikzstyle{selected edge} = [draw,line width=1.2pt,dashed,black]
\tikzstyle{ignored edge} = [draw,line width=1.2pt,black]

\path (0:1cm)    node[neighbor] (v1) {};
\path (72:1cm)   node[neighbor] (v2) {};
\path (2*72:1cm) node[neighbor] (v3) {};
\path (3*72:1cm) node[neighbor] (v4) {};
\path (4*72:1cm) node[neighbor] (v5) {};

\path[ignored edge] (v1) -- (v2);
\path[ignored edge] (v2) -- (v3);
\path[ignored edge] (v3) -- (v4);
\path[ignored edge] (v4) -- (v5);
\path[ignored edge] (v1) -- (v5);
\end{tikzpicture}}
\hspace{1.2cm}
\subfigure[Wheel]{
\begin{tikzpicture}[scale=1.2]
\tikzstyle{vertex}=[draw,shape=circle];
\tikzstyle{selected}=[vertex,fill=black];
\tikzstyle{neighbor}=[vertex,fill=black!30];
\tikzstyle{nonneighbor}=[vertex,fill=black!0];

\tikzstyle{edge} = [draw,thick,-]
\tikzstyle{selected edge} = [draw,line width=1.2pt,dashed,black]
\tikzstyle{ignored edge} = [draw,line width=1.2pt,black]

\path[selected] (0:0cm)    node[neighbor] (v0) {};
\path (0:1cm)    node[neighbor] (v1) {};
\path (72:1cm)   node[neighbor] (v2) {};
\path (2*72:1cm) node[neighbor] (v3) {};
\path (3*72:1cm) node[neighbor] (v4) {};
\path (4*72:1cm) node[neighbor] (v5) {};


\path[ignored edge] (v0) -- (v1);
\path[ignored edge] (v0) -- (v2);
\path[ignored edge] (v0) -- (v3);
\path[ignored edge] (v0) -- (v4);
\path[ignored edge] (v0) -- (v5);
\end{tikzpicture}}
\hspace{1.2 cm}
\subfigure[Grid]{
\begin{tikzpicture}[scale=1.2]
\tikzstyle{vertex}=[draw,shape=circle];
\tikzstyle{selected}=[vertex,fill=black];
\tikzstyle{neighbor}=[vertex,fill=black!30];
\tikzstyle{nonneighbor}=[vertex,fill=black!0];

\tikzstyle{edge} = [draw,thick,-]
\tikzstyle{selected edge} = [draw,line width=1.2pt,dashed,black]
\tikzstyle{ignored edge} = [draw,line width=1.2pt,black]


\node[neighbor] (v1) at (0,0) {};
\node[neighbor] (v2) at (0,1) {};
\node[neighbor] (v3) at (0,-1) {};
\node[neighbor] (v4) at (1,0) {};
\node[neighbor] (v5) at (-1,0) {};
\node[neighbor] (v6) at (-1,1) {};
\node[neighbor] (v7) at (1,-1) {};
\node[neighbor] (v8) at (1,1) {};
\node[neighbor] (v9) at (-1,-1) {};

\path[ignored edge] (v1) -- (v2);
\path[ignored edge] (v1) -- (v3);
\path[ignored edge] (v1) -- (v4);
\path[ignored edge] (v1) -- (v5);
\path[ignored edge] (v5) -- (v6);
\path[ignored edge] (v5) -- (v9);
\path[ignored edge] (v2) -- (v6);
\path[ignored edge] (v2) -- (v8);
\path[ignored edge] (v4) -- (v8);
\path[ignored edge] (v4) -- (v7);
\path[ignored edge] (v3) -- (v7);
\path[ignored edge] (v3) -- (v9);
\end{tikzpicture}}
\hspace{1.2cm}
\subfigure[Complete graph]{
\begin{tikzpicture}[scale=1.3]
\tikzstyle{vertex}=[draw,shape=circle];
\tikzstyle{selected}=[vertex,fill=black];
\tikzstyle{neighbor}=[vertex,fill=black!30];
\tikzstyle{nonneighbor}=[vertex,fill=black!0];

\tikzstyle{edge} = [draw,thick,-]
\tikzstyle{selected edge} = [draw,line width=1.2pt,dashed,black]
\tikzstyle{ignored edge} = [draw,line width=1.2pt,black]

\path (0:1cm)    node[neighbor] (v1) {};
\path (72:1cm)   node[neighbor] (v2) {};
\path (2*72:1cm) node[neighbor] (v3) {};
\path (3*72:1cm) node[neighbor] (v4) {};
\path (4*72:1cm) node[neighbor] (v5) {};

\path[ignored edge] (v1) -- (v2);
\path[ignored edge] (v1) -- (v3);
\path[ignored edge] (v1) -- (v4);
\path[ignored edge] (v1) -- (v5);
\path[ignored edge] (v2) -- (v3);
\path[ignored edge] (v2) -- (v4);
\path[ignored edge] (v2) -- (v5);
\path[ignored edge] (v3) -- (v4);
\path[ignored edge] (v4) -- (v3);
\path[ignored edge] (v4) -- (v5);
\end{tikzpicture}}
\caption{A depiction of communication networks used in simulations.}
\label{fig:various_network}
\end{figure}

For every type of connectivity graph,
we initiate the algorithm from a random starting point and terminate it
after $\tilde k$ iterations. A 95\% confidence interval 
of the mean sample error at the termination is computed for a sample of size 50, where 
the sample-path error is defined as in~\eqref{eqn:def-err_ell}.
The players' stepsize rules that we use are:
\[
\alpha_{k,i} =
\left \{
\begin{array}{ll}
\frac{9}{\Gamma_k(i)} & \mbox{for a diminishing stepsize} \\
 \alpha_i & \mbox{for a constant stepsize,}
\end{array}
\right.
\]
where $\alpha_i$ is randomly drawn from a uniform distribution, 
$\alpha_i\sim U(5e$-$3,1e$-$2)$. 
We again investigate cases when
there are 20 and 50 players in the network and derive the
	following insights:
\begin{itemize}
\item In Tables~\ref{tab:gossip1}--\ref{tab:gossip2}, we report the mean error 
and in Tables~\ref{tab:gossip_conf1}--\ref{tab:gossip_conf2}, we report the width 
of the confidence interval for various levels of 
$\tilde k$. The results are consistent with our 
theoretical findings, and they indicate a decrease in the mean 
error and the width of the
confidence interval with increasing $\tilde k.$ As 
expected, the mean error at the termination increases with
the size of the network. It is worth 
mentioning the discrepancy in the value of $\tilde k$ 
across synchronous and asynchronous algorithm. Note that
in the asynchronous algorithm, only two agents are 
performing updates and thus, for a network of size $N,$ 
$\tilde k$ global iterations translates to $2\tilde k /N$ iterations per agent, approximately.
\item On comparing the performance of the synchronous algorithm 
(cf.~Tables~\ref{tab:diffusion_err}--\ref{tab:diffusion_conf_int_static}) to that of the asynchronous algorithm
(cf.~Tables~\ref{tab:gossip1}--\ref{tab:gossip_conf2}), we observe that the
synchronous algorithm performs better than its asynchronous
counterpart in terms of mean error and the confidence width at termination. This is expected as 
in the synchronous setting, the players' communicate more frequently and the network diffuses information
faster than in the asynchronous setting.
\item The nature of the connectivity graph plays an important role in
	the performance 
of the synchronous algorithm. However, such an influence in the
asynchronous setting is less pronounced. 
\item In an effort to better understand the impact of connectivity, in
Table~\ref{tab:gossip_iter}, we compare the number of iterations\footnote{The iteration number is the mean for 50 sample rounded to the smallest integer over-estimate.} 
required for the player's to concur on the aggregate $\bar g^*$ within a threshold of 
1$\texttt{e}$-3 when the network consists of $N=20$ players. 
We also present a metric of connectivity density given by
$p_{\min}/p_{\max}$ as well as the 
square root of the second largest eigenvalue of the expected 
weight matrix, i.e., $\sqrt{\lambda},$ 
which in effect determines the rate of information dissemination in the network.
We note that the number of
iterations needed to achieve the threshold error correlates  with 
the value of $\sqrt{\lambda}$ and this prompts us to arrive at 
the following 
conclusion: Having a well-informed up-to-date neighbor is more 
important than having a denser connectivity. For instance, a
wheel network has a poor connectivity of all the network type based on
the $p_{\min}/p_{\max}$ criterion 
yet it has superior aggregate convergence to all but the complete 
network. In part, this is because the central agent in such a
	network updates throughout the course of the algorithm, allowing
		for good 
mixing of network wide information. In contrast, the cycle network 
though better connected yet cannot ensure good mixing of information,
	   given that no agent has access to ``good information.''
Similarly, a complete network 
provides each agent with an opportunity to communicate with every other 
agent and thus ensures good mixing of information. The grid 
network falls between the wheel and the cycle network in terms of
availability of well-informed neighbors and thus the performance.
\end{itemize}

%
%
\begin{table}[H]
\centering
\caption{Mean error after $\tilde k= 5e4$ iterations for gossip algorithm}
\label{tab:gossip1}
\begin{tabular}{c|c|c|c|c||c|c|c|c}
\hline
 & \multicolumn{4}{c||}{Constant Step} & \multicolumn{4}{c}{Diminishing Step}\\
\hline
N & Cycle & Wheel & Grid & Complete & Cycle & Wheel & Grid & Complete  \\
\hline
20	&  2.29$e$-3	& 3.66$e$-5	& 3.66$e$-5	& 3.66$e$-5	
&2.51$e$-2	& 1.01$e$-4	& 1.93$e$-3	& 4.64$e$-5\\
\hline
50	& 2.80$e$-1	& 6.76$e$-2	& 1.76$e$-1	& 1.26$e$-3	
&1.22	& 2.33$e$-2	& 8.41$e$-1	& 3.68$e$-3\\
\hline 
\end{tabular}
\end{table}

\begin{table}[H]
\centering
\caption{Mean error after $\tilde k= 1e5$ iterations for gossip algorithm}
\label{tab:gossip2}
\begin{tabular}{c|c|c|c|c||c|c|c|c}
\hline
 & \multicolumn{4}{c||}{Constant Step} & \multicolumn{4}{c}{Diminishing Step}\\
\hline
N & Cycle & Wheel & Grid & Complete & Cycle & Wheel & Grid 
& Complete\\
\hline
20	& 3.78$e$-5	& 3.66$e$-5	& 3.66$e$-5	& 3.66$e$-5	
& 3.93$e$-3	& 3.65$e$-5	& 1.69$e$-4	& 3.67$e$-5
\\ 
\hline
50 & 1.65$e$-1	& 1.09$e$-3	& 9.19$e$-2	& 6.23$e$-5	
& 7.63$e$-1	& 1.99$e$-3	& 4.57$e$-1	& 2.83$e$-4\\ 
\hline
\end{tabular}
\end{table}

\begin{table}[H]
\centering
\caption{Width of \jk{90\%} confidence interval after $\tilde k= 5e4$ iterations for gossip algorithm}
\label{tab:gossip_conf1}
\begin{tabular}{c|c|c|c|c||c|c|c|c}
\hline
 & \multicolumn{4}{c||}{Constant Step} & \multicolumn{4}{c}{Diminishing Step}\\
\hline
N & Cycle & Wheel & Grid & Complete & Cycle & Wheel & Grid & Complete  \\
\hline
20	& 1.87$e$-4	& 8.22$e$-7	& 5.34$e$-7	& 0$e$0	
& 7.51$e$-2	& 4.76$e$-3	& 2.08$e$-2	& 3.23$e$-3\\
\hline
50	& 1.68$e$-2	& 6.33$e$-3	& 1.89$e$-2	& 1.77$e$-4	
& 5.23$e$-1	& 7.23$e$-2	& 4.35$e$-1	& 2.87$e$-2\\
\hline 
\end{tabular}
\end{table}

\begin{table}[H]
\centering
\caption{Width of \jk{90\%} confidence interval after $\tilde k= 1e5$ iterations for gossip algorithm}
\label{tab:gossip_conf2}
\begin{tabular}{c|c|c|c|c||c|c|c|c}
\hline
 & \multicolumn{4}{c||}{Constant Step} & \multicolumn{4}{c}{Diminishing Step}\\
\hline
N & Cycle & Wheel & Grid & Complete & Cycle & Wheel & Grid 
& Complete\\
\hline
20	& 2.4$e$-6 & 4.74$e$-10	& 4.74$e$-10	& 4.74$e$-10	
& 4.40$e$-4	& 7.23$e$-7	& 2.07$e$-5	& 1.56$e$-7
\\ 
\hline
50 & 1.35$e$-2 & 1.33$e$-4	& 1.40$e$-2	& 1.78$e$-8	
& 6.91$e$-2	& 1.81$e$-4	& 4.15$e$-2	& 2.22$e$-5\\ 
\hline
\end{tabular}
\end{table}

\begin{table}[H]
\centering
\caption{Number of iteration for concurrence of player's aggregate
within an error of 1$e$-3} 
\label{tab:gossip_iter}
\begin{tabular}{c|c|c|c}
\hline
Network & $p_{\min}/p_{\max}$ & $\lambda$ & Iterations \\
\hline
Cycle & 1 & 0.9994 & 48818 \\
\hline
Wheel & 1/19 & 0.1622 & 8324 \\
\hline
Grid & 5/7 & 0.3151 & 17950 \\
\hline
Complete & 1 & 1.0888e-08 & 5842\\
\hline 
\end{tabular}
\end{table}

\section{Summary and Conclusions}\label{sec:conclusion}
This paper focuses on a class of Nash games in which player interactions
are seen through the aggregate sum of all players' actions. The players
and their interactions are modeled as a network with limited
connectivity which only allows for restricted local communication. We
propose two types of algorithms, namely, \us{a} synchronous (consensus-based)
	and an asynchronous (gossip-based) distributed algorithm, both of
	which abide \us{by an} information exchange restriction for computation of
	an equilibrium point. Our synchronous algorithm allows for
	implementation in a dynamic network with \us{a} time-varying
	connectivity \us{graph}.
	In contrast, our asynchronous algorithm allows for an implementation
	in a static network.  We establish error bounds on the deviation of
	players's decision from the equilibrium decision when a constant,
	yet player specific, stepsize is employed in the asynchronous
	algorithm.  Our extensions allow the players' decisions to be
	coupled in a more general form of ``aggregates''.  The contribution
	of our work can broadly be summarized as: (1) the development of
	synchronous and asynchronous distributed algorithms for aggregative
	games over graphs; (2) the establishment of the convergence of the
	algorithms to an equilibrium point, including the case with player
	specific stepsizes; and (3) an extension to a more general classes
	of aggregative games.  We also provide illustrative numerical
	results that support our theoretical findings.

\section*{Appendix}

{\bf Proof of Lemma~\ref{lemma:stepsize}:}
\begin{proof}
Note that $\Gamma_{k,i} =\sum_{t=1}^k \mathbbm{1}_{S_{i,t}} $ where 
$S_{i,t}$ is the event that
agent $i$ updates at time $t$ and $\mathbbm{1}_S$ is the indicator function of 
an event $S$. Since the events $S_{i,t}, t=1,2,\ldots,$ are i.i.d.\ with mean \us{$p_i$} 
for each $i \in \Nscr$, by the law of iterated logarithms (cf.~\cite{dudley}, pages 476--479), we have for any
$q > 0$, with probability 1,
$$\lim_{k\to \infty} \frac{|\Gamma_k(i) - kp_i|}{k^{\frac{1}{2}+q}} = 0 \; 
\mbox{for all} \; i \in \Nscr.$$
Thus, \uvs{for almost every $\omega \in \Omega$, there exists a  sufficiently large
	$\tilde k(\omega)$} (depending on $q$ and $\Nscr$), \uvs{such that} 
\begin{equation}\label{eq:m}
\frac{|\Gamma_k(i) - kp_i|}{k^{\frac{1}{2}+q}} \le \frac{1}{N^2} \; 
\mbox{for all} \; k \geq \tilde k(\uvs{\omega}) \; \mbox{and} \; i \in \Nscr.
\end{equation}
\uvs{This implies that for almost every $\omega \in \Omega$} and for all $i \in
\Nscr$ and $k \geq \tilde k(\uvs{\omega}),$
\begin{align}
\Gamma_k(i) \geq k p_i - \frac{1}{N^2}k^{\frac{1}{2}+q}=  
\left(p_ik^{\frac{1}{2}-q} - \frac{1}{N^2}\right)k^{\frac{1}{2}+q}.
\label{lem-71}
\end{align}
Now, let $q\in(0, 1/2)$, \uvs{implying that} the term $k^{\frac{1}{2}-q}p_i$ tends to infinity as $k$ tends to infinity.
Thus, we can choose a larger $\tilde k\uvs{(\omega)}$ (if needed) so that with
probability 1,  
we have
\begin{align}
p_ik^{\frac{1}{2}-q} - \frac{1}{N^2} \geq \frac{1}{2}p_ik^{\frac{1}
{2}-q} \qquad \mbox{for all } k \geq \tilde k \mbox{ and } i \in \Nscr, 
	\label{lem-72}
\end{align}
\uvs{or equivalently
\[
\mathbb{P} \left[ \omega: p_ik^{\frac{1}{2}-q} - \frac{1}{N^2} \geq \frac{1}{2}p_ik^{\frac{1}
{2}-q} \quad \mbox{for all } k \geq \tilde k(\uvs{\omega}) \mbox{ and }
i \in \Nscr \right] = 1. 
\]
}
From \uvs{\eqref{lem-71} and \eqref{lem-72}}, we obtain that \uvs{for
	almost every $\omega \in \Omega$},  $\Gamma_k(i) \geq \frac{1}
{2}k p_i$ for all $k\ge\tilde k\uvs{(\omega)}$ and all $i \in \Nscr$, implying that
\begin{equation}\label{eq:mm}
\uvs{ \mathbb{P}\left[\omega : \frac{1}{\Gamma_k(i)} \leq \frac{2}{kp_i} \quad \mbox {for all } k \geq 
\tilde k(\uvs{\omega}) \mbox{ and } i \in \Nscr\right] = 1},
\end{equation}
thus showing the first relation of the lemma  holds in view of $\alpha_{k,i} =\frac{1}{\Gamma_k(i)}.$


We now consider $\left|\alpha_{k,i} - \frac{1}{kp_i} \right|.$ \uvs{For
almost	every $\omega \in \Omega$}, we have 
	for all $k\geq \tilde k(\uvs{\omega})$ and $i \in \Nscr$,
\[
\left| \alpha_{k,i} - \frac{1}{kp_i}  \right| = \frac{1}
{kp_i}\frac{1}{\Gamma_k(i)}|kp_i - \Gamma_k(i)|\leq \frac{2}
{k^2p_i^2}|kp_i - \Gamma_k(i)|,
\]
where the inequality follows by~\eqref{eq:mm}. By using relation~\eqref{eq:m}, it follows that
\uvs{for almost every $\omega \in \Omega$} and for all $k \geq \tilde k(\uvs{\omega})$ and $i \in \Nscr,$
\[
\left|\alpha_{k,i} - \frac{1}{kp_i} \right| \leq \frac{2}{k^2p^2_{i}}\frac{k^{\frac{1}{2}+q}}{N^2} 
=\frac{2}{k^{\frac{3}{2} -q} p^2_{i}N^2}.
\]
Since agent $i$ updates with probability $p_i = \frac{1}{N}(1+ \sum_{j \in 
\Nscr_i} p_{ij}),$ it follows that
\begin{equation}
p_i \geq \frac{1}{N} \left(1+ \left(\min_{\{i,j\} \in \Escr} p_{ij}\right) |\Nscr_i| \right)\geq 
\frac{1}{N}\left(1 + \min_{\{i,j\} \in \Escr} p_{ij}\right), 
\end{equation}
where the last inequality follows from $|\Nscr_i| \geq 1$ (since the graph 
$(\Nscr,\us{\Escr})$ has no isolated node). 
Hence, 
\uvs{for almost every $\omega \in \Omega$} and for all $k\ge \tilde
	k(\uvs{\omega})$ and all $i$,
\[
\left|\alpha_{k,i} - \frac{1}{kp_i} \right| \leq \frac{2N^2}{k^{\frac{3}{2} -q} \left(1 + \min_{\{i,j\} \in \Escr} p_{ij}\right)^2 N^2}
=\frac{2}{k^{\frac{3}{2} -q} \left(1 + \min_{\{i,j\} \in \Escr} p_{ij}\right)^2}.
\]
The desired relation follows by letting $\hat p=1 + \min_{\{i,j\} \in
	\Escr} p_{ij}$:
\uvs{\[
\mathbb{P}\left[\omega: \left|\alpha_{k,i} - \frac{1}{kp_i} \right| \leq
\frac{2}{k^{\frac{3}{2} -q} \left(1 + \min_{\{i,j\} \in \Escr}
		p_{ij}\right)^2} \mbox{ for } k \geq \tilde{k}(\omega)\right] =
1.
\]}

\end{proof}

\section*{Acknowledgement}
The authors are deeply grateful to A. Kulkarni and B. Touri in providing some helpful suggestions regarding the proof of Lemma 10.

\bibliographystyle{abbrv}
\bibliography{user_optim,references}

\end{document}